\documentclass[12pt,a4paper]{article}

\usepackage{amssymb}
\usepackage{amsmath, amsthm}
\usepackage{arydshln}
\usepackage{epsfig}
\usepackage{setspace}

\usepackage{geometry}

\def\RR{\mathbb{R}}
\def\ZZ{\mathbb{Z}}
\def\QQ{\mathbb{Q}}
\def\FF{\mathbb{F}}
\def\KK{\mathbb{K}}

\numberwithin{equation}{section}

\newcommand{\rank}{\mathop{\rm rank} }

\newcommand{\Det}{\mathop{\rm Det} }
\newcommand{\ncrank}{\mathop{\rm nc\mbox{-}rank} }

\newtheorem{Thm}{Theorem}[section]
\newtheorem{Prop}[Thm]{Proposition}
\newtheorem{Lem}[Thm]{Lemma}
\newtheorem{Cor}[Thm]{Corollary}
\theoremstyle{definition}

\newtheorem{Rem}[Thm]{Remark}
\newtheorem{Ex}[Thm]{Example}

\title{Computing the degree of determinants via discrete convex optimization on Euclidean buildings}
\author{Hiroshi HIRAI \\
Department of Mathematical Informatics, \\
Graduate School of Information Science and Technology,   \\
The University of Tokyo, Tokyo, 113-8656, Japan.\\
\texttt{\normalsize hirai@mist.i.u-tokyo.ac.jp}
}
\begin{document}

\maketitle
\begin{abstract}
	In this paper, we consider the computation of 
	the degree of the Dieudonn\'e determinant of
	a linear symbolic matrix 
	$
	A = A_0 + A_1 x_1 + \cdots + A_m x_m,
	$
	where each $A_i$ is an $n \times n$ polynomial matrix 
	over $\KK[t]$ and
	$x_1,x_2,\ldots,x_m$ are pairwise ``non-commutative" variables.
	This quantity is regarded as a weighted generalization of 
	the non-commutative rank (nc-rank) of a linear symbolic matrix, and its computation
	is shown to be a generalization of several basic combinatorial optimization problems, 
	such as weighted bipartite matching and weighted linear matroid intersection problems.

	Based on the work on nc-rank by Fortin and Rautenauer (2004), and Ivanyos, Qiao, and Subrahmanyam (2018), we develop a framework to compute 
	the degree of the Dieudonn\'e determinant of a linear symbolic matrix.
	We show that the deg-det computation reduces to 
	a discrete convex optimization problem on 
	the Euclidean building for ${\rm SL}(\KK(t)^n)$.
	To deal with this optimization problem, 
	we introduce a class of discrete convex functions on the building.
	This class is a natural generalization of L-convex functions in 
	discrete convex analysis (DCA).
	We develop a DCA-oriented algorithm (steepest descent algorithm) 
	to compute the degree of determinants.
    Our algorithm works with matrix computation on $\KK$, and uses a subroutine to compute 
    a certificate vector subspace for the nc-rank, 
    where the number of calls of the subroutine is sharply estimated.
    Our algorithm enhances 
    some classical combinatorial optimization algorithms with new insights, and
    is also understood as a variant of 
    the combinatorial relaxation algorithm,
    which was developed  earlier by Murota
    for computing the degree of the (ordinary) determinant.
\end{abstract}


\section{Introduction}
A {\em linear symbolic matrix} or {\em linear matrix} $A$ 
is a matrix each of whose entries is a linear (affine) function in variables $x_1,x_2,\ldots, x_m$. 
Namely $A$ admits the form of
\begin{equation}\label{eqn:A}
A = A_0 + A_1 x_1 + \cdots + A_m x_m,
\end{equation}
where each $A_i$ is a matrix over a field $\mathbb{K}$ and 
$A$ is viewed as a matrix over the polynomial ring 
$\mathbb{K}[x_1,x_2,\ldots,x_m]$ or the rational function field $\mathbb{K}(x_1,x_2,\ldots,x_m)$. 
In this paper, we address the symbolic rank computation of linear matrices
and its generalization.
This problem is a fundamental problem in discrete mathematics and computer science.
In a classical paper~\cite{Edmonds67} in combinatorial optimization,
Edmonds noticed that 
the maximum matching number of a bipartite graph
is represented as the rank of 
such a matrix: For each edge $e = ij$, 
introduce a variable $x_e$ and a matrix $E_{e}$ having $1$ for 
the $(i,j)$-entry and zero for the others. 
Then the rank of linear matrix $A = \sum_{e} E_e x_e$ 
(with  $A_0 = 0$ in (\ref{eqn:A})) is equal to 
the maximum matching number of the graph.
In the same paper, Edmonds asked for a polynomial time algorithm
to compute the rank of a general linear matrix.
Clearly the rank computation is easy if 
we substitute an actual number for each variable.  
According to random substitution,
Lov\'asz~\cite{Lovasz79} developed 
a randomized polynomial time algorithm to compute the rank of linear matrices.
Designing a deterministic polynomial time algorithm is 
a big challenge in theoretical computer science, 
since it would lead to a breakthrough in circuit complexity theory~\cite{KabanetsImpagliazzo04}. 
Currently, such deterministic algorithms are known 
for very restricted classes of linear matrices.
Most of them are connected to 
polynomially-solvable combinatorial optimization problems, 
such as matching, matroid intersection, 
and their generalizations; see~\cite{Lovasz89}.

The rank of linear matrices also plays important roles in engineering applications.
A {\em mixed matrix}~due to Murota and Iri~\cite{MurotaIri85} 
is a linear matrix including constant matrix $A_0$ in the above bipartite graph example.
Mixed matrices are applied to the analysis of a linear control system 
including physical parameters (such as positions, temperatures) that cannot be measured exactly, 
where these parameters are modeled as variables $x_i$.
Then the rank describes several fundamental characteristics (such as the controllability) of the system.
See~\cite{MurotaMatrix} for the theory of mixed matrices and its applications.
The (infinitesimal) {\em rigidity} of a bar-and-joint structure 
is also characterized by 
the rank of a linear matrix, called the {\em rigidity matrix},  
where variables $x_i$ represent ``generic" positions of the structure.
Polynomial-time rank computation is 
known for mixed matrices (see \cite{MurotaMatrix}) and 
some classes of rigidity matrices (see e.g.,~\cite{Lovasz89}). 

Recently there are significant developments in 
the rank computation of linear matrices. 
In the above paragraph,
we assumed that $A$ 
is a matrix over polynomial ring $\KK[x_1,x_2,\ldots, x_m]$ 
with (commutative) indeterminates $x_1,x_2,\ldots,x_m$, 
and the rank is considered in rational function field $\KK(x_1,x_2,\ldots, x_m)$. 
However $A$ can  also be viewed as a matrix over the free ring $\KK\langle x_1,x_2,\ldots,x_m \rangle$ generated by $x_1,x_2,\ldots,x_m$, 
where variables are supposed to be pairwise non-commutative, i.e., $x_i x_j \neq x_j x_i$. 
It is shown by Amitsur~\cite{Amitsur1966} that 
there is a non-commutative analogue $\KK(\langle x_1,x_2,\ldots,x_m \rangle)$ of the rational function field, called the {\em free skew field}, to which
$\KK\langle x_1,x_2,\ldots,x_m \rangle$ is embedded. 
Now we can define the rank of $A$ over this skew field 
$\KK(\langle x_1,x_2,\ldots,x_m \rangle)$.
This rank concept of $A$ is called 
the {\em non-commutative rank} ({\em nc-rank}) of $A$.
Fortin and Rautenauer~\cite{FortinReutenauer04} proved 
a formula of the nc-rank, which says that the nc-rank is equal to the optimal value of 
an optimization problem over the lattice of all vector subspaces of $\KK^n$.
Garg, Gurvits, Oliveira, and Wigderson~\cite{GGOW15} 
proved that the nc-rank of $A$ can be computed in deterministic 
polynomial time if $\KK = \QQ$.
They showed that Gurvits' {\em operator scaling algorithm}~\cite{Gurvits04}, 
which was earlier developed for the rank computation of 
a special class (called {\em Edmonds-Rado class}) of linear matrices, 
can be a polynomial time algorithm for the nc-rank. 
Ivanyos, Qiao, and Subrahmanyam~\cite{IQS15a,IQS15b}
developed a polynomial time algorithm to compute the nc-rank 
over an arbitrary field. 
Their algorithm is  
a vector-space analogue of the augmenting path algorithm for bipartite matching,
and utilizes an invariant theoretic result 
by Derksen and Makam~\cite{DerksenMakam17} for the complexity estimate. 
Independent of this line of research, 
Hamada and Hirai~\cite{HamadaHirai17} investigated (a variant of) 
the optimization problem for the nc-rank which they called 
the {\em maximum vanishing subspace problem (MVSP)}.
Their motivation comes from a canonical form of 
a matrix under block-restricted transformations~\cite{ItoIwataMurota94,IwataMurota95}. 
They also developed a polynomial time algorithm for the nc-rank
based on the fact that MVSP is viewed as a submodular function optimization 
on the modular lattice of all vector subspaces of $\KK^n$.

In this paper, 
we consider a ``weighted analogue" of the nc-rank.
Our principal motivation is 
to capture the weighted versions of combinatorial optimization problems
from the non-commutative points of view.
Consider, for example, 
the weighted matching problem on a bipartite graph, where
two color classes are supposed to have the same cardinality, 
and each edge $e$ has (integer) weight $c_e$. 
By introducing new indeterminate~$t$, 
modify the above bipartite-graph linear matrix $A$ as 
$A := \sum_{e} t^{c_e} x_{e} E_e$.
Then the maximum weight of a perfect matching
is equal to the degree of the determinant of $A$ with respect to $t$.
This well-known example suggests that 
 such a weighted analogue is the degree of the determinant.
This motivates us to consider the degrees of the determinants of linear matrices 
in the non-commutative setting.

The main contribution of this paper is to develop 
a computational framework for  
the degrees of the determinants in the non-commutative setting,  
which captures some of classical weighted combinatorial optimization problems. 
Our results and their features are summarized as follows:
\begin{itemize}
	\item For a determinant concept for matrices over skew field $\FF$, 
	     we consider the {\em Dieudonn\'e determinant}~\cite{Dieudonne1943}.
	     Although the value of the Dieudonn\'e determinant is 
	     no longer an element of the ground field, 
	     in the skew field $\FF(t)$ of rational functions ({\em Ore quotient ring} of polynomial ring $\FF[t]$), 
	     its degree is well-defined, see e.g., \cite{Taelman06}.
	     In this paper, we use the notation $\Det$ 
	     for the Dieudonn\'e determinant, whereas  
	     $\det$ is used for the ordinary determinant.
	     
	     Our target is 
	     the degree $\deg \Det A$ of the Dieudonn\'e determinant $\Det A$ 
	     of a linear matrix $A = A_0+ A_1 x_1 + \cdots  + A_m x_m$, 
	     where each $A_i$ is a square polynomial matrix over $\KK[t]$
	     and $A$ is viewed as a matrix over
	     the rational function skew field $\FF(t)$ of the free skew field 
	     $\FF = \KK(\langle x_1,x_2,\ldots,x_m \rangle)$.

	\item  We establish a duality theorem for $\deg \Det$, 
	which is a natural generalization of the Fortin-Rautenauer formula for the nc-rank.
	In fact, 
	a weak duality relation was previously observed by Murota~\cite{Murota95_SICOMP} for $\deg \det$, and is now a strong duality for $\deg \Det$. 
    Analogously to the Fortin-Rautenauer formula saying that
	the nc-rank is equal to the optimal value of an 
	optimization problem (MVSP) over the lattice of all vector subspaces of $\KK^n$,  
	our formula says that $\deg \Det$ is equal to the optimal value of 
	an optimization problem over the lattice 
	of all full-rank $\KK(t)^-$-submodules of $\KK(t)^n$, 
	where $\KK(t)^-$ is the valuation ring of $\KK(t)$ 
	with valuation $\deg$.
	In the literature of group theory, 
	this lattice structure is known as the {\em Euclidean building} for ${\rm SL}(\KK(t)^n)$~\cite{BruhatTits}, 
	whereas the lattice of all vector subspaces is the 
	{\em spherical building} for ${\rm SL}(\KK^n)$~\cite{Tits}.
	
	\item  We approach this optimization problem on the Euclidean building 
	from {\em Discrete Convex Analysis (DCA)}~\cite{MurotaBook} 
	with its recent generalization~\cite{HH17survey,HH16L-convex}. 
	Although DCA was originally a theory of discrete convex functions on $\ZZ^n$ that generalizes matroids and submodular functions, 
	recent study~\cite{HH17survey,HH16L-convex}
	shows that DCA-oriented concepts and algorithm design 
	are effective and useful for optimization problems 
	on certain discrete structures beyond $\ZZ^n$.
   {\em L-convexity}, which is one of the central concepts of DCA, is particularly important for us.
	L-convex functions are generalization of submodular functions
	and arise naturally from representative combinatorial 
	optimization problems 
	such as minimum-cost network flow and weighted bipartite matching.
	L-convex functions admit a simple minimization algorithm, 
	called the {\em steepest descent algorithm (SDA)}, 
	on which our algorithm for deg Det will be built. 

	We introduce an analogue of an L-convex function 
	on the building. 
	The previous work~\cite{HH17survey,HH16L-convex} introduced L-convexity on Euclidean buildings of type C, whereas our building here is of type A. 
	We show that the established formula of $\deg \Det$ 
	gives rise to an L-convex function, 
	analogously to the submodular function in MVSP for the nc-rank. 
	Consequently $\deg \Det$ is computed via an 
	L-convex function minimization on the Euclidean building.

	\item We develop an algorithm to compute $\deg \Det A$ for linear polynomial matrices $A$ 
	over $\KK[t]$. 
	Our algorithm requires a subroutine to solve MVSP over $\KK$, and is described in terms of matrix computation over $\KK$. 
	However it can be viewed as the steepest descent algorithm (SDA) 
	applied to the L-convex function on the Euclidean building.
	Geometrically, SDA traces the 1-skeleton of the building 
	with decreasing the value of the L-convex function.
	Each move is done by solving an optimization problem 
	on the spherical building associated with the local structure of the Euclidean building.
	This local problem coincides with MVSP.  
	By utilizing the recent analysis on SDA~\cite{MurotaShioura14}, 
	we show that the number of the moves
	is sharply estimated by the {\em Smith-McMillan form} of $A$.
	
	Our algorithm can also be interpreted as a variant of
	the {\em combinatorial relaxation algorithm},  
	which was  developed earlier for 
	$\deg \det$ of matrices (without variables) by Murota~\cite{Murota90_SICOMP,Murota95_SICOMP} and was further extended to mixed polynomial matrices by Iwata and Takamatsu~\cite{IwataTakamatsu13}; see~\cite[Section 7.1]{MurotaMatrix} 
	and recent work~\cite{IwataOkiTakamatsu17}.
	This interpretation sheds building-theoretic insights on 
	the combinatorial relaxation algorithm.
	
	\item We study a class of linear matrices $A$ for which 
	$\deg \det A = \deg \Det A$ holds.
	In the case of the nc-rank, 
	it is known from the results in~\cite{IvanyosKarpinskiOiaoSantha15, Lovasz79} that if each $A_i$ other than $A_0$
	is a rank-1 matrix over $\KK$, then it holds 
	$\rank A = \ncrank A$.
	We show a natural extension: if each $A_i$ other than $A_0$ 
	is a rank-1 matrix over $\KK(t)$, then it holds $\deg \det A = \deg \Det A$.
	This property implies that some of
	classical combinatorial optimization problems, represented as $\deg \det$, 
	fall into our framework of $\deg \Det$.
	Examples include weighted bipartite matching and weighted linear matroid intersection.
	In these examples, the optimal value is interpreted as $\deg \det$ as well as
	$\deg \Det$.
	A {\em mixed polynomial matrix}~\cite{MurotaMatrix} is also such an example.
	We explain how our SDA framework works for these examples, and discuss connections to some of classic algorithms, such as the Hungarian method,  the matroid greedy algorithm, and the matroid intersection algorithm by Lawler~\cite{Lawler75} and Frank~\cite{Frank81}. 
	For mixed polynomial matrices,
	our framework brings a new algorithm, which is faster 
	than the previous one~\cite{IwataOkiTakamatsu17,IwataTakamatsu13} 
	in terms of time complexity.
	One of motivating applications of mixed polynomial matrices is 
	the analysis of {\em differential algebraic equations (DAE)}; see \cite[Chapter 6]{MurotaBook}. 
	We present a possible application of our result 
	to the mixed-matrix DAE analysis.
\end{itemize}

The rest of this paper is organized as follows.
In Section~\ref{sec:skewfield}, 
we summarize basic facts on skew field,
nc-rank, Fortin-Rautenauer formula, MVSP, 
and Dieudonn\'e determinant.
In Section~\ref{sec:L-convex}, 
we introduce L-convex functions on Euclidean buildings and
show their basic properties.
Instead of dealing with Euclidean buildings in the usual axiom system, 
we utilize an elementary lattice-theoretic equivalent concept, 
called {\em uniform modular lattices}~\cite{HH18a}.
This class of modular lattices admits the L-convexity concept 
in a straightforward way.
In Section~\ref{sec:computing},
we establish a formula for the degree of the Dieudonn\'e determinant of a linear matrix 
and present an algorithm.
In Section~\ref{sec:rank-1}, we study 
linear matrices with rank-1 summands.

Closing the introduction, let us mention a recent work
by Kotta, Belikov, Hal\'as and Leibak~\cite{KottaBelikovHalasLeibak2017} 
in control theory. 
They showed that the order of the minimal state-space realization 
of a (linearized) nonlinear control system
is described by the degree of the Dieudonn\'e determinant 
of the non-commutative polynomial matrix associated with the system,  
where the matrix involves elements in 
the skew polynomial ring of non-commuting variables $t,\delta$.
This result generalizes the basic fact on the degrees of 
the (ordinary) determinants in linear time-invariant control systems; 
see e.g., \cite{Kailath80}.
They gave a primitive algorithm to compute $\deg \Det$ 
based on (symbolic) Gaussian elimination.
It would be an interesting research direction 
to extend our framework to deal with this type of deg-Det computation.

\section{Skew field}\label{sec:skewfield}

A {\em skew field} (or {\em division ring}) is a ring $\FF$ 
such that every nonzero element $x \in \FF$ 
has inverse element $x^{-1} \in \FF$ with $x^{-1}x = x x^{-1} = 1$.
The product $\FF^n$ of $\FF$ will be treated as 
a right $\FF$-vector space
of column vectors as well as a left $\FF$-vector space of row vectors; 
which one we suppose will be clear in the context.
The set of all $n \times n'$ matrices over $\FF$ is denoted by $\FF^{n \times n'}$.
The {\em rank} of matrix $A \in \FF^{n \times n'}$ is defined as the dimension of 
the right $\FF$-vector space spanned by columns of $A$, 
which is equal to the dimension of the left $\FF$-vector space spanned by rows of $A$.
Let $\ker_{\rm R} A$ denote 
the right kernel $\{x \in {\FF}^{n'} \mid Ax = {\bf 0}\}$, and 
let $\ker_{\rm L} A$ denote 
the left kernel $\{x \in {\FF}^n \mid x A = {\bf 0}\}$. 
Then the rank of $A \in \FF^{n \times n'}$ 
is equal to $n - \dim \ker_{\rm L} A = n' - \dim \ker_{\rm R} A$.
A square matrix $A \in \FF^{n \times n}$ is called {\em nonsingular} 
if its rank is equal to $n$, or equivalently 
if it has the inverse, which is denoted by $A^{-1}$, i.e., $A A^{-1} = A^{-1} A = I$.
These properties are easily seen from the Bruhat normal form of $A$;
see Lemma~\ref{lem:Bruhat} in Section~\ref{subsec:Det}.

An $n \times n'$ matrix $A \in \FF^{n \times n'}$ 
is viewed as a map $\FF^n \times \FF^{n'} \to \FF$ by
\begin{equation}\label{eqn:bilinear}
A(x,y) := x A y \quad (x \in \FF^n, y \in \FF^{n'}).
\end{equation}
Then $A$ is bilinear in the sense that $A(\alpha x +\alpha' x', y) 
= \alpha A(x, y) + \alpha' A(x', y)$ and $A(x, y\beta + y'\beta') = A(x,y)\beta +   A(x,y')\beta'$. Conversely, 
any bilinear map on the product of a left $\FF$-vector space $U$ and a right $\FF$-vector space $V$
is identified with a matrix over $\FF$ by choosing bases of $U$ and $V$.    
Let ${\cal S}_{\rm R}(\FF^n)$ and ${\cal S}_{\rm L}(\FF^n)$
denote the families of all right and left $\FF$-vector subspaces of $\FF^n$, respectively.
If $\FF$ is commutative, then $\rm R$ and $\rm L$ are omitted, such as ${\cal S}(\FF^n)$.

\subsection{Free skew field, nc-rank, and MVSP}
Let $\KK$ be a field.
Let $\KK[x_1,x_2,\ldots,x_m]$ and 
$\KK(x_1,x_2,\ldots,x_m)$ denote the ring of polynomials and 
the field of rational functions with variables $x_1,x_2,\ldots,x_m$, respectively, 
where variables are supposed to commute each other, i.e., $x_i x_j = x_j x_i$.
Let $\KK \langle x_1,x_2,\ldots,x_m \rangle$ be 
the free ring generated by pairwise non-commutative variables $x_1,x_2,\ldots,x_m$ over $\KK$.
It is known that the free ring $\KK \langle x_1,x_2,\ldots,x_m \rangle$
is embedded to the universal skew field of fractions, 
called the {\em free skew field}, which 
is denoted by $\KK (\langle x_1,x_2,\ldots,x_m \rangle)$, or simply, by $\KK(\langle x \rangle)$.
Elements of $\KK (\langle x_1,x_2,\ldots,x_m \rangle)$ 
are equivalence classes of all rational expressions
constructed from $x_i$, $x_i^{-1}$, and elements of $\KK$, 
under an equivalence relation 
obtained by substitutions of nonsingular matrices for variables $x_i$; see~\cite{Amitsur1966,CohnAlgebra3,CohnSkewField} for details.
We do not go into the detailed construction
of $\KK (\langle x_1,x_2,\ldots,x_m \rangle)$.

 As mentioned in the introduction,  a {\em linear symbolic matrix} or {\em linear matrix} 
on $\KK$ 
is a matrix of form $A = A_0 + A_1 x_1 + A_2 x_2 + \cdots + A_m x_m$, 
where each summand $A_i$ is a matrix over $\KK$.
A linear matrix is viewed as a matrix over $\KK(x_1,x_2,\ldots,x_m)$ 
as well as over $\KK(\langle x_1,x_2,\ldots,x_m \rangle)$.
%
The (commutative) rank of $A$ is defined as the rank of $A$ 
regarded as a matrix over $\KK(x_1,x_2,\ldots,x_m)$.
The {\em non-commutative rank} ({\em nc-rank}) of $A$, denoted by nc-$\rank A$, is defined 
as the rank of $A$ regarded as a matrix 
over $\KK(\langle x_1,x_2,\ldots,x_m \rangle)$.
The nc-rank is not less than the (commutative) rank.

Let $A = A_0 + A_1 x_1 + A_2 x_2 + \cdots + A_m x_m$ 
be a linear $n \times n'$ matrix.
We consider an upper bound of nc-$\rank A$. 
For nonsingular matrices $S \in \KK^{n \times n},T \in \KK^{n'  \times n'}$ (not containing variables), 
if $SAT$ has a zero submatrix of $r$ rows and $s$ columns, 
then nc-$\rank A$ is at most $n + n' - r - s$.
This gives rise to the following optimization problem (MVSP):
\begin{eqnarray*}
{\rm MVSP}: \quad {\rm Max.} && r + s  \\
{\rm s.t.} && \mbox{$SAT$ has a zero submatrix of $r$ rows and $s$ columns},\\
  && S \in \KK^{n \times n}, T \in \KK^{n' \times n'}: \mbox{nonsingular}.
\end{eqnarray*}
This upper bound was observed by Lov\'asz~\cite{Lovasz89} for the usual rank.
The name MVSP becomes clear below.
Fortin and Reutenauer~\cite{FortinReutenauer04} 
showed that this upper bound is actually tight for  the nc-rank.
\begin{Thm}[\cite{FortinReutenauer04}]\label{thm:FortinReutenauer}	
	For a linear $n \times n'$ matrix $A = A_0 + A_1 x_1 + A_2 x_2 + \cdots + A_m x_m$, 
	{\rm nc}-$\rank A$ is equal to $n+n'$ minus the optimal value of MVSP.
\end{Thm}
As mentioned in the introduction,
Garg, Gurvits, Oliveira, and Wigderson~\cite{GGOW15} showed that 
nc-$\rank A$ can be computed in polynomial time 
when $\KK = \QQ$.
One drawback of this algorithm is not to output optimal matrices $S,T$ in MVSP.  
Such an algorithm was developed by
Ivanyos, Qiao, and Subrahmanyam~\cite{IQS15a,IQS15b} for an arbitrary field. 
\begin{Thm}[\cite{IQS15a,IQS15b}]
	MVSP can be solved in polynomial time.
\end{Thm}
Hamada and Hirai~\cite{HamadaHirai17} also gave a ``polynomial time" algorithm  
based on submodular optimization (see Lemma~\ref{lem:submo}), 
though the bit-length required in the algorithm is not bounded if $\KK = \QQ$.
Note that their formulation is slightly different from that presented here; 
see Section~\ref{subsec:HH} in Appendix for the relation.

As \cite{HamadaHirai17,HH16DM} did, 
MVSP is formulated as 
the following optimization problem 
({\em maximum vanishing subspace problem}) 
of the space of all vector subspaces of $\KK^n$: 
\begin{eqnarray}
	{\rm MVSP}: \quad {\rm Max.} && \dim X + \dim Y  \nonumber \\
	{\rm s.t.} && A_i (X,Y) = \{0\} \quad (i=0,1,2,\ldots,m),\nonumber \\
	&& X \in {\cal S}(\KK^n), Y \in {\cal S}(\KK^{n'}). \label{eqn:MVSP}
\end{eqnarray}
Recall the notation ${\cal S}(\KK^n)$ for all vector subspaces of $\KK^n$, and that
each matrix $A_i$ is viewed as a bilinear form $\KK^n \times \KK^{n'} \to \KK$ 
as in (\ref{eqn:bilinear}).
We call an optimal $(X,Y)$ a {\em maximum vanishing subspace} or an {\em mv-subspace}.  
Notice that we can eliminate variable $X$ by substituting 
$X = A(Y)^{\bot}$ ($=$ the orthogonal complement of the image of $Y$ by $A$), 
and obtain the formulation of~\cite{GGOW15,IQS15a,IQS15b} --- 
the {\em minimum shrunk subspace problem}.

The nc-$\rank A$ is also obtained via MVSP over $\KK(\langle x \rangle)$:
\begin{eqnarray*}
	\overline{\rm MVSP}: \quad {\rm Max.} && \dim X + \dim Y  \\
	{\rm s.t.} && A (X,Y) = \{0\},\\
	&&  X \in {\cal S}_{\rm L}(\KK(\langle x\rangle)^n), 
	Y \in {\cal S}_{\rm R}(\KK(\langle x\rangle)^{n'}).
\end{eqnarray*}
Indeed, $\overline{\rm MVSP}$ has obvious optimal solutions
$( \KK(\langle x\rangle)^n,  \ker_{\rm R} A)$ and $(\ker_{\rm L} A, \KK(\langle x\rangle)^{n'})$.
Notice that $\KK(\langle x \rangle)^n$ is  a scalar extension 
of $\KK^n$, i.e.,
$\KK(\langle x \rangle)^n \simeq \KK(\langle x \rangle) \otimes \KK^n$.
Therefore a feasible solution in MVSP 
is embedded into a feasible solution in
$\overline{\rm MVSP}$ by $(X,Y) \mapsto (\KK(\langle x \rangle) 
\otimes X, Y\otimes \KK(\langle x \rangle))$.
Then Theorem~\ref{thm:FortinReutenauer} also says that 
MVSP is an exact inner approximation of $\overline{\rm MVSP}$:
\begin{Lem}\label{lem:innerapprox}
	Any mv-subspace of MVSP is also an mv-subspace of $\overline{\mbox{MVSP}}$.
\end{Lem}

\subsection{Dieudonn\'e determinant}\label{subsec:Det}
Here we introduce a determinant concept for matrices over skew field $\FF$, known as {\em Dieudonn\'e determinant}~\cite{Dieudonne1943}.
Our reference of Dieudonn\'e determinant is \cite[Section 11.2]{CohnAlgebra3}.
Our starting point is 
the following normal form of matrices over skew field $\FF$. 
\begin{Lem}[{Bruhat normal form; 
		see \cite[THEOREM 2.2 in Section 11.2]{CohnAlgebra3}}]\label{lem:Bruhat}
Any matrix $A$ over $\FF$ is represented as
\begin{equation}\label{eqn:Bruhat}
A = L D P U,
\end{equation}
for a diagonal matrix $D$, a permutation matrix $P$, 
a lower-unitriangular matrix $L$, 
and an upper-unitriangular matrix $U$, 
where $D P$ is uniquely determined.
\end{Lem}
Here a {\em lower(upper)-unitriangular matrix} is a lower(upper)-triangular matrix such that all diagonals are $1$.
The proof is done by the Gaussian elimination, and is 
essentially  the LU-decomposition (without pivoting).

Let $\FF^{\times} := \FF \setminus \{0\}$ denote the multiplicative group of $\FF$, and let $[\FF^{\times}, \FF^{\times}]$ denote the derived group of $\FF^{\times}$, i.e.,
$[\FF^{\times}, \FF^{\times}]$ is the group generated by all commutators $aba^{-1}b^{-1}$.
The {\em abelianization} $\FF^{\times}_{\rm ab}$ of $\FF^{\times}$ is defined 
by $\FF^{\times}_{\rm ab} := \FF^{\times} / [\FF^{\times}, \FF^{\times}]$.
For a nonsingular matrix $A$, 
the {\em Dieudonn\'e determinant} ${\rm Det} A$ of $A$ is defined by
\begin{equation}
{\rm Det} A :=   {\rm sgn} (P) d_1 d_2 \cdots d_n \mod [\FF^{\times}, \FF^{\times}], 
\end{equation}
where $A$ is represented as (\ref{eqn:Bruhat}) 
for permutation matrix $P$ and diagonal matrix $D$ with nonzero diagonals  $d_1,d_2,\ldots,d_n$.
If $\FF$ is a field, 
then $[\FF^\times,\FF^\times] = \{1\}$,  $\FF^{\times}_{\rm ab} = \FF^{\times}$, and $\Det = \det$. 
\begin{Ex}\label{ex:2x2}
	Consider the case of a $2$ by $2$ matrix
	$\left(
	\begin{array}{cc}
		a & b \\
		c & d 
	\end{array}
	\right)$.
	If $a \neq 0$, then the Bruhat normal form is given by
	\[
	\left(
	\begin{array}{cc}
	a & b \\
	c & d 
	\end{array}
	\right) = 	\left(
	\begin{array}{cc}
	1 & 0 \\
	ca^{-1} & 1 
	\end{array}
	\right) \left( \begin{array}{cc}
	a & 0 \\
	0 & d - c a^{-1}b
	\end{array}
	\right)\left(
	\begin{array}{cc}
	1 & a^{-1}b \\
	0 & 1 
	\end{array}\right).
	\]
	Also, if $a = 0$ and $b \neq 0$, then
	\[
	\left(
	\begin{array}{cc}
	a & b \\
	c & d 
	\end{array}
	\right) = 	\left(
	\begin{array}{cc}
	1 & 0 \\
	d b^{-1} & 1 
	\end{array}
	\right) \left( \begin{array}{cc}
	b & 0 \\
	0 & c
	\end{array}
	\right)\left(
	\begin{array}{cc}
	0 & 1 \\
	1 & 0 
	\end{array}\right)
	\left(
	\begin{array}{cc}
	1 & 0 \\
	0 & 1 
	\end{array}\right).
	\]
	Hence the Dieudonn\'e determinant is given by  
	\begin{equation}
	\Det 
	\left(
	\begin{array}{cc}
	a & b \\
	c & d 
	\end{array}
	\right)
	= \left\{ \begin{array}{ccc}
	a(d -  c a^{-1}b) & \mbox{mod $[\FF^{\times}, \FF^{\times}]$} & {\rm if}\ a \neq 0, \\
	- bc & \mbox{mod $[\FF^{\times}, \FF^{\times}]$} & {\rm if}\ a = 0.
	\end{array}
	\right. 
	\end{equation}
\end{Ex}
The Dieudonn\'e determinant has the desirable properties
of the determinant, though its value is no longer an element of $\FF$.
\begin{Lem}[{\cite{Dieudonne1943}; see \cite[Theorem 2.6 in Section 11.2]{CohnAlgebra3}}]\label{lem:DetAB}
	For nonsingular matrices $A,B \in \FF^{n \times n}$, 
	it holds $\Det A B = \Det A \Det B$.
\end{Lem}

Next 
we introduce the polynomial ring $\mathbb{F}[t]$ 
and its skew field $\mathbb{F}(t)$ of fractions 
(or the {\em Ore quotient ring} of $\FF[t]$), where
the indeterminate $t$ commutes every element in~$\mathbb{F}$.
Here $\mathbb{F}[t]$ consists of polynomials 
$\sum_{i=0}^k a_i t^i$, where $k \geq 0$ and $a_i \in \FF$. 
By using commuting rule $x t = t x$, 
the addition and multiplication in $\FF[t]$ are naturally defined.
The resulting ring $\mathbb{F}[t]$ is called a polynomial ring 
over $\mathbb{F}$ with indeterminate $t$.  In the notation of~\cite{CohnAlgebra3,CohnSkewField,NoetherianBook},   
$\mathbb{F}[t]$ is the skew polynomial ring $\FF[t; 1,0]$.
The {\em degree} $\deg p$ of polynomial $p = \sum_{i=0}^k a_i t^i$ with $a_k \neq 0$
is defined by $\deg p := k$.

The polynomial ring $\mathbb{F}[t]$ is an  {\em Ore domain}, i.e., 
any two nonzero polynomials $p,q \in \FF[t]$ 
admit a common multiple $pu = q v$ for some nonzero $u,v \in \mathbb{F}[t]$.
See \cite[Section 9.1]{CohnAlgebra3} and 
\cite[Chapter 6]{NoetherianBook} for the details of Ore domains 
and their fields of fractions. 
This enables us to introduce addition and multiplication on
the set $\mathbb{F}(t)$ of all fractions $p/q$ for $p \in \mathbb{F}[t], q \in \mathbb{F}[t] \setminus \{0\}$.
Here $p/q$ is the equivalence class of $(p,q) \in \mathbb{F}[t] \times  \mathbb{F}[t] \setminus \{0\}$ 
under the equivalence relation: $(p,q) \sim (p',q')$ $\Leftrightarrow$ 
$(pu,qu) = (p'v,q'v)$ for some nonzero $u,v \in \mathbb{F}[t]$.
Addition $p/q + p'/q'$ is defined as $(pu + p'v)/qu$ by choosing 
$u,v \in \mathbb{F}[t]$ with $qu= q'v$. 
Multiplication $(p/q)(p'/q')$ is defined as $pu/q'v$ by 
choosing $u,v \in \mathbb{F}[t]$ with $qu = p'v$.
They are well-defined.
The inverse of nonzero element $p/q$ (i.e., $p \neq 0$) is given by $q/p$.
In this way, $\mathbb{F}(t)$ becomes a skew field 
into which $\mathbb{F}[t]$ is embedded by $p \mapsto p/1$.
Element $1/t^k$, which is denoted by $t^{-k}$, commutes each element of $\mathbb{F}(t)$.

The degree $\deg p/q$ of $p/q$ is defined as $\deg p - \deg q$.
As was observed by~Taelman~\cite{Taelman06}, 
the degree of the Dieudonn\'e determinant is well-defined, 
since the degree is zero on commutators.
\begin{Ex}\label{ex:2x2-deg} We see from Example~\ref{ex:2x2} that
	the degree of the determinant of a $2 \times 2$ matrix over $\mathbb{F}(t)$ is similar to the commutative case:
	\begin{equation*}
	\deg \Det 
	\left(
	\begin{array}{cc}
	a & b \\
	c & d 
	\end{array}
	\right) \leq \max\{\deg(a) + \deg(d), \deg(b) + \deg(c) \}.
	\end{equation*}
	The equality holds if $\deg(a) + \deg(d) \neq \deg(b) + \deg(c)$.
\end{Ex}
We let $\deg \Det A := - \infty$ if $A$ is singular.
From the definition and Lemma~\ref{lem:DetAB}, we have:
\begin{Lem}[see \cite{Taelman06}]\label{lem:DegAB}
	For $A,B \in \mathbb{F}(t)^{n \times n}$, 
	it holds $\deg \Det AB = \deg \Det A + \deg \Det B$.
\end{Lem}

By using the Dieudonn\'e determinant, we can 
formulate the {\em Smith-McMillan form} of a matrix over $\mathbb{F}(t)$; 
see~\cite[Section 5.1.2]{MurotaMatrix} for the commutative case.
An element $p/q \in \mathbb{F}(t)$ is said to be {\em proper} if $\deg(p/q) \leq 0$.
Let $\mathbb{F}(t)^-$ denote the ring of proper elements of $\mathbb{F}(t)$. 
A matrix over $\mathbb{F}(t)^-$ is also called proper. 
A proper matrix is called {\em biproper} if it is nonsingular 
and its inverse is also proper.
For integer vector $\alpha \in \ZZ^n$, let $(t^{\alpha})$ 
denote the diagonal $n \times n$ matrix such that the $(i,i)$-entry is $t^{\alpha_i}$ for $i=1,2,\ldots,n$. 
\begin{Prop}[Smith-McMillan form]\label{prop:SmithMcMillan}
For a nonsingular matrix $A \in \mathbb{F}(t)^{n \times n}$, 
there are biproper matrices $S,T$ and integer vector $\alpha \in \ZZ^n$ with $\alpha_1 \geq \alpha_{2} \geq \cdots \geq \alpha_n$ such that
\begin{equation*}
S A T = (t^{\alpha}).
\end{equation*} 
The integers $\alpha_k$ are uniquely determined by
\begin{equation*}
\alpha_k = \delta_k - \delta_{k-1} \quad (k = 1,2,\ldots,n),
\end{equation*}
where $\delta_k$ is the maximum degree of the Dieudonn\'e determinants 
of $k \times k$ submatrices of $A$, and let $\delta_0 := 0$.
\end{Prop}
A part of the statement is given as an exercise in 
\cite[pp. 459--460]{CohnAlgebra3}
in a general setting of a valuation ring.
The proof is given in Appendix~\ref{app:SM}, 
which goes in almost the same way as in the commutative case.

Any proper element $p/q$ 
is written as $u + p'/q$, where $u \in \mathbb{F}$ and $\deg p' < \deg q$.
Indeed, if $p/q = (a t^k + p'')/(b t^k + q')$ for $a,b \in \FF$, $b \neq 0$, $\deg p'' < k$, and $\deg q' < k$, then $u = ab^{-1}$.
This element $u$ is uniquely determined (independent of expression $p/q$); 
see~\cite[Exercise 6F]{NoetherianBook}.
Thus any proper matrix $A$ is uniquely written as $A = A^0 + t^{-1} A'$, 
where $A^0$ is a matrix over $\mathbb{F}$ and $A'$ is proper.
\begin{Lem}\label{lem:key}
	Let $A$ be a square proper matrix.
	Then the degree of each diagonal of the Smith-McMillan form of $A$ is nonpositive, and $\deg  \Det A \leq 0$. 
	In addition, the following conditions are equivalent:
	\begin{itemize}
		\item[{\rm (1)}] $\deg  \Det A = 0$.
		\item[{\rm (2)}] $A^0$ is nonsingular over $\FF$.
		\item[{\rm (3)}] $A$ is biproper.
		\item[{\rm (4)}] $A$ is written as $Q_1 Q_2 \cdots Q_k$ $(k \geq 0)$, 
		where each $Q_i$ is a permutation matrix, proper unitriangular matrix, or 
		diagonal matrix with degree-zero elements. 
	\end{itemize}
\end{Lem}
\begin{proof}
	The former part is immediate from Lemma~\ref{prop:SmithMcMillan} with
	the fact that $\alpha_1$ is the maximum degree of entries of $A$, and is now nonpositive; then $\deg \Det A = \delta_n = \sum_{k=1}^n \alpha_k \leq 0$. 
	
	We show the equivalence.
	
	(4) $\Rightarrow$ (3) follows from the fact that 
	each $Q_i$ is biproper.
	
	(3) $\Rightarrow$ (2). 
	If $B$ is the inverse of $A$ and is represented as 
	$B = B^0 + t^{-1}B'$ for proper $B'$, 
	then $B^0$ must be the inverse of $A^0$.
	
	(2) $\Rightarrow$ (1).  Consider the Smith-McMillan form 
	$A = S(t^{\alpha})T$.  From $\alpha \leq 0$ and $S^0 (t^{\alpha})^0 T^0 = A^0$, 
	if $\deg \Det A = \sum_{k} \alpha_k <0$, then $\alpha_k < 0$ for some $k$, and
	$A^0$ must be singular.
	  
	 (1) $\Rightarrow$ (4).  We see in the proof of Proposition~\ref{prop:SmithMcMillan} (Appendix~\ref{app:SM}) 
	 that $S,T$ in the Smith-McMillan form of $A$
	 are taken as the product of those matrices.
\end{proof}
\begin{Lem}
	For $A \in \mathbb{F}(t)^{n \times n}$, it holds $\deg \Det A \leq n \alpha_1$.
	In addition, if $A$ is a nonsingular polynomial matrix, then $\deg \Det A \geq 0$. 
\end{Lem}
The latter part is contained in \cite[Theorem 1.1]{Taelman06}.
\begin{proof}
	The first statement follows from the Smith-McMillan form 
	and $\deg \Det S = \deg \Det T = 0$ for biproper matrices $S,T$ 
	(Lemma~\ref{lem:key}).
	The polynomial ring $\FF[t]$ is a (left and right) Euclidean domain.
	Therefore, by elementary row and column  operations on $\FF[t]$ 
	with row and column permutations,
	$A$ is diagonalized so that the diagonal entries are polynomials in $\FF[t]$ 
	(such as the Smith form); see \cite[Section 9.2]{CohnAlgebra3}. 
	Namely $PAQ$ is a diagonal polynomial matrix
	for some matrices $P,Q$ with $\deg \Det P = \deg \Det Q = 0$.
	This implies $\deg \Det A = \deg \Det PAQ \geq 0$.
\end{proof}

Finally we note a useful discrete convexity property of the degree of the Dieudonn\'e determinant.
A {\em valuated matroid}~\cite{DressWenzel_greedy,DressWenzel} on a set $E$ is a function $\omega: 2^E \to \RR \cup \{-\infty\}$ satisfying the following condition:
\begin{description}
	\item[{\rm (EXC)}] For any $X,Y \subseteq E$ with $\omega(X),\omega(Y) \neq - \infty$ 
	and $e \in X \setminus Y$, there is $f \in Y \setminus X$ such that
	\begin{equation}
	\omega(X) + \omega(Y) \leq \omega(X \cup \{f\} \setminus \{e\}) + \omega(Y \cup \{e\} \setminus \{f\}). 
	\end{equation}
\end{description} 
It is well-known in the literature that 
the deg-det function gives rise to a valuated matroid~\cite{DressWenzel_greedy,DressWenzel}; see~\cite[Chapter 5]{MurotaMatrix}.
This is also the case for the deg-Det function.
\begin{Prop}\label{prop:valuated}
	Let $A$ be an $n \times m$ matrix over $\FF(t)$. 
	The following function $\omega_A: 2^{\{1,\ldots,m\}} \to \RR \cup \{-\infty\}$ is 
	a valuated matroid:
    \begin{equation}
    \omega_A(X) 
    := \left\{ \begin{array}{ll}
    \deg \Det A[X] & {\rm if}\ |X| = n, \\
    - \infty & {\rm otherwise}
    \end{array}\right.
    \quad (X \subseteq \{1,2,\ldots,m\}),
    \end{equation}
    where $A[X]$ denotes the submatrix of $A$ consisting of the $i$-th columns over $i \in X$.
\end{Prop}
\begin{proof} 
	For the verification we use a local characterization \cite[Theorem 5.2.25]{MurotaMatrix} of valuated matroids, which says that function $\omega:2^{\{1,2,\ldots,m\}} \to \RR \cup \{-\infty\}$ is 
	a valuated matroid if and only if $\{ X \subseteq \{1,2,...,m\} \mid \omega(X) \neq - \infty\}$ 
	is the base family of a matroid 
	and $\omega$ satisfies (EXC) for all pairs $X,Y$ 
	with $|X \setminus Y| = |Y \setminus X| = 2$.
	The first condition for $\omega_A$ follows from the observation that 
    $\omega_A(X) \neq \infty$ if and only 
	if $A[X]$ is nonsingular, i.e., 
	$X$ forms a basis of (right) vector space $\FF(t)^n$ over skew field $\FF(t)$.
	Therefore the family $\{ X \subseteq \{1,2,...,m\} \mid \omega(X) \neq - \infty\}$ 
	is the base family of a representable matroid over skew field $\FF(t)$. 

	We next consider the latter condition, i.e.,  
	(EXC) for $X,Y$ with $|X \setminus Y| = |Y \setminus X| = 2$, where
	we can assume that $\omega_{A}(X) \neq -\infty$.
	Let $A'$ denote the $n \times (n+4)$ submatrix of $A$ consisting columns in $X \cup Y$.
	By Lemma~\ref{lem:DegAB}, 
	elementary row operations (or multiplying a nonsingular matrix from left) to $A$ 
	does not change $\omega_A$ other than constant addition.
	Also we can arrange columns of $A'$ so that $X \cap Y$  forms the first $n-2$ columns. 
	Hence we can assume that $A'$ is the form $\left( \begin{array}{cc} I & C \\ O & B \end{array} \right)$,
	where $I$ is the unit matrix of size $n-2$, $B$ is a $2 \times 4$ matrix, 
	and  $C$ is an $(n-2) \times 4$ matrix. 
	Then $\omega_{A'} ( (X \cap Y) \cup \{i, j\}) = \omega_B (\{i,j\})$ holds 
	for $i,j \in (X \setminus Y) \cup (Y \setminus X)$.
	This can be seen from the definition of the Dieudonn\'e determinant as: 
	By elementary column operations (or multiplying an upper-unitriangular 
	matrix from right), $A[(X \cap Y) \cup \{i,j\}]$ becomes 
	$\left( \begin{array}{cc} I & O \\ O & B[\{i,j\}] \end{array} \right)$, 
	and then $\Det A[(X \cap Y) \cup \{i,j\}] = \Det B[\{i,j\}]$.

	Therefore this reduces our problem to the verification of (EXC) for an arbitrary $2 \times 4$ matrix $A$.
	Then (EXC) is equal to: 
	\begin{description}
		\item[{\rm (4PT)}] the maximum of $\omega(12) + \omega(34)$, $\omega(13) + \omega(24)$, $\omega(14) + \omega(23)$ is attained at least twice,
	\end{description}
	where $\omega_A (\{i,j\})$ is simply written as $\omega(ij)$.

	We may assume that $\omega(12) \neq - \infty$ and the $(1,1)$-entry is nonzero (by column permutation). 
	By row operations, we can make $A$ so that the $(2,1)$-entry is zero. 
	If the $(2,3)$-entry is nonzero, 
	then make $A$ so that $(1,3)$-entry is zero.
	Then we may consider two cases:
	\begin{equation*}
	\left( \begin{array}{cccc}
	a & c & d & e  \\
	0 & b & 0 & f
	\end{array}
	\right) \ {\rm and}\  
	\left( \begin{array}{cccc}
	a & c & 0 & e  \\
	0 & b & d & f
	\end{array}
	\right). 
	\end{equation*}
	Recall Example~\ref{ex:2x2-deg}.
	For the former case, $\omega(12) + \omega(34) = \omega(14) + \omega(23) = \deg(a) + \deg (b) + \deg(d) + \deg (f)$, and $\omega(13) + \omega(24) = - \infty$.
	For the latter case, 
	$\omega(12) + \omega(34) = \deg(a) + \deg (b) + \deg(d) + \deg (e)$,
	$\omega(14) + \omega(23) = \deg(a) + \deg (f) + \deg(c) + \deg (d)$, 
	and $\omega(13) + \omega(24) \leq \deg (a) + \deg (d) +\max \{ \deg (c) + \deg(f), \deg(e) + \deg(b)\}$ with equality if $\deg (c) + \deg(f) \neq \deg(e) + \deg(b)$.
	Thus (4PT) holds for all cases.
\end{proof}

\section{L-convex function on Euclidean building}\label{sec:L-convex}

In this section, we introduce 
L-convex function on Euclidean building. 
It will turn out in Section~\ref{sec:computing} that the deg-Det computation 
reduces to an L-convex function minimization 
on the Euclidean building for ${\rm SL}(\KK(t)^n)$.
Our approach is lattice-theoretic. 
First we set up basic lattice terminologies.
Then we introduce a uniform modular lattice, 
which is a lattice-theoretic counterpart of 
a Euclidean building of type A, 
and we introduce L-convexity on it.

\subsection{Lattice}
 
A {\em lattice} is a partially ordered set ${\cal L}$
such that 
every pair $x,y$ of elements has the minimum common upper bound
$x \vee y$ and the maximum common lower bound $x \wedge y$; the former is called the {\em join} and the latter is called the {\em meet}.
The partial order is denoted by $\preceq$, 
where $x \prec y$ is meant as $x \preceq y$ and $x \neq y$.
A totally ordered subset of ${\cal L}$ is called a {\em chain}, 
which is written as $x^0 \prec x^1 \prec \cdots \prec x^k$.
The {\em length} of chain $C$ is defined as $|C|-1$.
For $x,y \in {\cal L}$ with $x \preceq y$, 
the {\em interval} $[x,y]$ is defined as the set of elements $z$ 
with $x \preceq z \preceq y$.
If $[x,y] = \{x,y\}$, we say that $y$ {\em covers} $x$.
In this paper, we only consider lattices in which 
every chain of every interval has a finite length. 
A {\em sublattice} ${\cal L}'$ 
is a subset of ${\cal L}$ such that $x,y \in {\cal L}'$ 
implies $x \wedge y, x \vee y \in {\cal L}'$.
An interval is a sublattice.
The {\em opposite} $\check{\cal L}$ of lattice ${\cal L}$
is the lattice obtained from ${\cal L}$ 
by reversing the partial order of ${\cal L}$.
The direct product ${\cal L} \times {\cal L}'$ of 
two lattices ${\cal L},{\cal L}'$ becomes 
a lattice by the product order: 
$(x,x') \preceq (y,y')$ 
$\Leftrightarrow$ $x \preceq y$ and $x' \preceq y'$.

A {\em modular lattice} is a lattice ${\cal L}$
such that for every triple $x,y,z \in {\cal L}$ 
with $x \preceq z$, it holds $x \vee (y \wedge z) = (x \vee y) \wedge z$.
The opposite of a modular lattice  is also a modular lattice.
A useful criterion for the modularity is given as follows.
A {\em valuation} on a lattice ${\cal L}$ is a function 
$v: {\cal L} \to \RR$ such that
\begin{itemize}
\item $v(x) + v(y) = v(x \wedge y) + v(x \vee y)$ for all $x,y \in {\cal L}$, and
\item $v(x) < v(y)$ for all $x,y \in {\cal L}$ with $x \prec y$.
\end{itemize}
\begin{Lem}[\cite{Birkhoff}]
If a lattice ${\cal L}$ admits a valuation, then ${\cal L}$ is a modular lattice.
\end{Lem}
A {\em unit valuation} is a valuation $v$ such that 
$v(x) = v(y) - 1$ provided $y$ covers $x$. 
A modular lattice (having the minimum element) is said to be {\em complemented}
if every element is the join of atoms ($=$ elements covering the minimum element).
\begin{Ex}
	Let $\FF$ be a skew field. 
	The families ${\cal S}_{\rm R}(\FF^n)$ and ${\cal S}_{\rm L}(\FF^n)$ of 
	vector subspaces of $\FF^n$
	are complemented modular lattices, where the partial order is the inclusion relation.
	The join and meet are given by $+$ and $\cap$, respectively. 
	Also $X \mapsto \dim X$ is a unit valuation.
	
	The family of chains of 
	${\cal S}_{\rm R}(\FF^n) \setminus \{\emptyset,  \FF^n \}$ 
	(or ${\cal S}_{\rm L}(\FF^n) \setminus \{\emptyset,  \FF^n \}$)
	is known as the {\em spherical building} of ${\rm SL}(\FF^n)$. 
	More generally, the family of chains of a complemented modular lattice 
	is equivalent to a spherical building of type A. See~\cite{Tits}.
\end{Ex}

A function $f$ on lattice ${\cal L}$ is called {\em submodular} 
if it satisfies
\begin{equation*}
f(x) + f(y) \geq f(x \wedge y) + f(x \vee y) \quad (x,y \in {\cal L}). 
\end{equation*}
As was noticed in~\cite{HamadaHirai17,HH16DM, ItoIwataMurota94},  
MVSP is viewed as a submodular optimization on a modular lattice.
For a matrix $A \in \FF^{n \times n'}$ 
regarded as a bilinear form (\ref{eqn:bilinear}), 
define $r_A:{\cal S}_{\rm L}(\FF^{n}) 
\times {\cal S}_{\rm R}(\FF^{n'}) \to \ZZ$ by
\begin{equation*}
r_A(X,Y) := \mbox{the rank of the restriction of $A$ to $X \times Y$.}
\end{equation*}
\begin{Lem}[\cite{IwataMurota95}; also see \cite{HamadaHirai17,HH16DM}]\label{lem:submo}
    Let $A \in \FF^{n \times n'}$ be a matrix over $\FF$. 
	Then $r_{A}$ is submodular on ${\cal S}_{\rm L}(\FF^{n}) \times \check{\cal S}_{\rm R}(\FF^{n'})$, i.e.,
	\begin{equation*}
		r_A(X,Y) + r_A(X',Y') \geq 
		r_A(X + X',Y \cap Y') + r_A(X \cap X',Y + Y'). 
	\end{equation*}
\end{Lem}
The vanishing condition $A(X,Y) = \{0\}$ is equivalent to $r_A(X,Y) = 0$. 
By including $r_A$ in the objective as a penalty term, 
MVSP is formulated as 
an unconstrained submodular optimization
over modular lattice ${\cal S}(\KK^{n}) \times \check{\cal S}(\KK^{n'})$:
\begin{eqnarray}
{\rm MVSP}': \quad {\rm Min.} && - \dim X - \dim Y + C \sum_{i=1}^m r_{A_i}(X,Y) \nonumber \\
{\rm s.t.} && X \in {\cal S}(\KK^n), Y \in {\cal S}(\KK^{n'}), \nonumber
\end{eqnarray}
where $C > 0$ is a large constant.
The approach by Hamada and Hirai~\cite{HamadaHirai17} is based on this idea.

\subsection{Uniform modular lattice and Euclidean building}\label{subsec:uniform}
The {\em ascending operator}  of a lattice ${\cal L}$ is 
a map $(\cdot)^+:{\cal L} \to {\cal L}$
defined by
\begin{equation}
(x)^+ := \bigvee \{ y \in {\cal L} \mid \mbox{$y$ covers $x$} \} \quad (x \in {\cal L}).
\end{equation}
A {\em uniform modular lattice}~\cite{HH18a} is a modular lattice ${\cal L}$ such that
the ascending operator is defined and is an automorphism on ${\cal L}$.
Suppose that ${\cal L}$ is a uniform modular lattice.  
The rank ($=$ the length of a maximal chain) of $[x, (x)^+]$ 
is independent of $x$, 
and is called the {\em uniform-rank} of ${\cal L}$.
The inverse of $(\cdot)^+$ is given by $x \mapsto$ 
the meet of elements covered by $x$. 
In particular, the opposite $\check{\cal L}$ of ${\cal L}$ is also uniform modular.
The product of two uniform modular lattices is also uniform modular.
\begin{Ex}
	$\ZZ^n$ becomes a lattice with respect to vector order $\leq$, 
	where $x \vee y$ equals $\max(x,y)$ (componentwise maximum of $x,y$) and 
	$x \wedge y$ equals $\min(x,y)$ (componentwise minimum of $x,y$).
	Now $\ZZ^n$ is a uniform modular lattice, 
	where $z \mapsto \sum_{i=1}^n z_i$ is a unit valuation,  
	the ascending operator is given by $x \mapsto x + {\bf 1}$ for all one-vector $\bf 1$, 
	and the uniform-rank is equal to $n$. 
\end{Ex}

A {\em $\ZZ^n$-skeleton} of ${\cal L}$ 
is a sublattice $\varSigma$ such that $\varSigma$ is isomorphic to $\ZZ^n$ and 
the restriction of the ascending operator of ${\cal L}$ to $\varSigma$ 
is the same as the ascending operator of $\varSigma$.
A chain $x_0 \prec x_1 \prec \cdots \prec x_m$ is said to be 
{\em short} if $x_m \preceq (x_0)^+$.
\begin{Lem}[{\cite{HH18a}}]\label{lem:skeleton}
	Let ${\cal L}$ be a uniform modular lattice with uniform-rank $n$.
	\begin{itemize}
		\item[{\rm (B1)}] For two short chains $C,D$, 
		there is a $\ZZ^n$-skeleton $\varSigma$ containing them.
		\item[{\rm (B2)}] If two $\ZZ^n$-skeletons $\varSigma,\varSigma'$ contain 
		short chains $C,D$, there is an order-preserving bijection from $\varSigma$ to $\varSigma'$ such that it is the identity on $C \cup D$.  
	\end{itemize}	
\end{Lem}
(B1) and (B2) are essentially 
the apartment axiom 
of {\em Euclidean building of type~A}~\cite{BruhatTits}; see \cite{Garrett}.
The paper~\cite{HH18a} shows that 
the family of all short chains
in a uniform modular lattice actually forms
a Euclidean building of type A,
and that every Euclidean building of type A is obtained in this way.
An {\em apartment system} of ${\cal L}$ is a family of $\ZZ^n$-skeletons 
such that a $\ZZ^n$-skeleton in (B1) 
can be chosen from the family.
A $\ZZ^n$-skeleton in an apartment system is simply called an {\em apartment}. 

Next we consider an important example of a uniform modular lattice
arising from a skew field with a discrete valuation.
Let $\FF$ be a skew field, and let $\FF(t)$ be 
the skew field of rational functions over $\FF$.
Let $\FF(t)^-$ be the ring of proper elements of $\FF(t)$. 
Consider the $n$-product $\FF(t)^n$, which is regarded 
as a left $\FF(t)^-$-module of row vectors as well as a right $\FF(t)^-$-module of column vectors.
Let ${\cal L}_{\rm L} (\FF(t)^n)$ 
denote the family of all full-rank free 
$\FF(t)^-$-submodules\footnote{In the literature of building, such a module is called a {\em lattice}. We do not use this term for avoiding confusion.} of $\FF(t)^n$,  where $\FF(t)^n$ is regarded as a left $\FF(t)^-$-module of row vectors.
Let ${\cal L}_{\rm R} (\FF(t)^n)$ be defined as the right analogue.
By definition, 
an element $L \in {\cal L}_{\rm L}(\FF(t)^n)$ is represented 
as $\langle Q \rangle_{\rm L} := \{ \lambda Q \mid \lambda \in (\FF(t)^-)^n \}$ for 
a nonsingular matrix $Q$ over $\FF(t)$.
Similarly,  an element $L \in {\cal L}_{\rm R}(\FF(t)^n)$ 
is written as $\langle P \rangle_{\rm R} := \{ P\lambda  \mid \lambda \in (\FF(t)^-)^n \}$ for 
a nonsingular matrix $P$ over $\FF(t)$.
For $L \in {\cal L}_{\rm L}(\FF(t)^n)$ or ${\cal L}_{\rm R}(\FF(t)^n)$, 
define $\deg L$ by
\begin{equation}
\deg L := \deg \Det P
\end{equation}
for a nonsingular matrix $P$ with $L = \langle P \rangle_{\rm L}$ or 
$\langle P \rangle_{\rm R}$.
This is well-defined; if $\langle P \rangle_{\rm R} = \langle P' \rangle_{\rm R}$, 
then $P' = P S$ for some biproper matrix $S$, and $\deg \Det P' = \deg \Det P$ by Lemmas~\ref{lem:DegAB} and \ref{lem:key}.

We give three lemmas on the family ${\cal L}_{\rm R}(\FF(t)^n)$ below.
They hold when $\rm R$ is replaced by $\rm L$.
The first one is shown in~\cite{HH18a} for the case where $\FF$ is a field.
\begin{Lem}[\cite{HH18a}]\label{lem:L(F(t)^n)}
	${\cal L}_{\rm R}(\FF(t)^n)$ is a uniform modular lattice, 
	where  $\wedge = \cap$, $\vee = +$, 
	$L \mapsto \deg L$ is a unit valuation, 
	the uniform-rank is equal to $n$, 
	and the ascending operator is given by $L \mapsto tL$.
\end{Lem}
We give in the appendix a proof by adapting the argument in \cite{HH18a} 
for our non-commutative setting.

For an integer vector $z \in \ZZ^n$, 
recall that $(t^z)$ denotes the diagonal matrix with diagonals $t^{z_1},t^{z_2},\ldots,t^{z_n}$. 
For a nonsingular matrix $Q$, let $\varSigma_{\rm R}(Q)$ denote 
the sublattice of ${\cal L}_{\rm R}(\FF^n)$ 
consisting of $\langle Q(t^z) \rangle_{\rm R}$ for all $z \in \ZZ^n$. Similarly,  define $\varSigma_{\rm L}(Q)$ by $\varSigma_{\rm L}(Q) := \{ \langle (t^z)Q \rangle_{\rm L} \mid z \in \ZZ^n \}$.
\begin{Lem}[{see \cite[Chapter 19]{Garrett} for the commutative version}]\label{lem:apartment}
	The family of sublattices 
	consisting of $\varSigma_{\rm R}(Q)$ 
	for all nonsingular $Q \in \FF(t)^{n \times n}$ forms an apartment system in ${\cal L}_{\rm R}(\FF(t))$, 
	where $z \mapsto  \langle Q(t^z) \rangle_{\rm R}$ is 
	an isomorphism between $\ZZ^n$ and $\varSigma_{\rm R}(Q)$.
\end{Lem}
The proof is given in the appendix.
Next we study the lattice structure of interval 
$[L, (L)^+] = [L,tL]$, which is a complemented modular lattice 
and is turned out to be the spherical building at the link of $L$.
For $M \in {\cal L}_{\rm R}(\FF(t)^n)$ with $L \subseteq M \subseteq tL$, 
the quotient module $M/L$ becomes a right $\FF$-vector space 
by $(u + L)\alpha := u \alpha + L$ for $\alpha \in \FF$. 
For an $\FF$-vector subspace $X$ of $tL/L$, 
define submodule $L \circ X$ of $tL$ by
\begin{equation}
L \circ X := \{ u \in tL \mid u + L \in X  \}.
\end{equation} 

\begin{Lem}\label{lem:[L,tL]}
	Let $L \in {\cal L}_{\rm R}(\FF(t)^n)$.
 \begin{itemize}
 	\item[{\rm (1)}] $tL/L$ is a right $\FF$-vector space with dimension $n$.
 	\item[{\rm (2)}] $[L,tL]$ is isomorphic to ${\cal S}_{\rm R}(tL/L)$ 
 	by $M \mapsto M/L$ with inverse $X \mapsto L \circ X$. 
 	\item[{\rm (3)}] For $X \in {\cal S}_{\rm R}(tL/L)$, 
 	it holds $\deg L \circ X = \deg L + \dim X$.
 	\item[{\rm (4)}] If $L = \langle P \rangle_{\rm R}$, then $[L,tL]$ is given by
 	\begin{equation*}
 	[L,tL] = \{ \langle P S (t^{{\bf 1}_{\leq k}}) \rangle_{\rm R} \mid 0 \leq k \leq n,\ S \in \FF^{n \times n}: \mbox{nonsingular} \}, 
 	\end{equation*}
 	where ${\bf 1}_{\leq k}$ denotes the 0,1-vector such that 
 	the first $k$ elements are $1$ and others are zero.  
\end{itemize}
\end{Lem}
\begin{proof}
	(1). Suppose that $\{ p_1,p_2,\ldots,p_n\}$ is a basis of $L$.
	Then $\{tp_1,tp_2,\ldots,tp_n\}$ is a basis of $tL$.
	We show that 
	$\{ tp_1 + L,tp_2+ L,\ldots,tp_n + L\}$ is a basis of $tL/L$.
	Every element $u \in tL$ is written as 
	 $u = \sum_{i=1}^n t p_i \lambda_i$ for $\lambda_i \in \FF(t)^-$. 
	Here $\lambda_i$ is written as 
	$\lambda_i = \lambda_i^0 + t^{-1} \lambda_i'$ for $\lambda_i^0 \in \FF$ and $\lambda_i' \in \FF(t)^-$.
	This means that $u  \in \sum_{i=1}^n t p_i \lambda_i^0 + L$. 
	Thus $\{tp_1 + L,tp_2+ L,\ldots,tp_n + L\}$ spans $tL/L$.
	We show the linear independence.
	Suppose that $\sum_{i=1}^n t p_i \alpha_i = u \in L$ for $\alpha_i \in \FF$.
    Since $\{p_1,p_2,\ldots,p_n\}$ is a basis of $L$, 
	it must be $\alpha_i t \in \FF(t)^-$, and $\alpha_i = 0$.
	
	(2). It suffices to verify that $L \circ X$ is a full-rank free submodule.
	Since $tL$ is a free $\FF(t)^-$-module and $\FF(t)^-$ is PID,  
	the submodule $L \circ X$ of $tL$ is a free module containing $L$, and hence has rank $n$.
	
	(3) follows from (2) and 
	the fact that $\deg$ and $\dim$ are unit valuations.
	
	(4). By the proof of (1), 
	the column vectors of $tP$ modulo $L$ become an $\FF$-basis of $tL/L$.
	Therefore any vector subspace $X \subseteq tL/L$ 
	is spanned by $\FF$-linear combinations of column vectors of $tP$ modulo $L$.
	Thus, if $\dim X =k$, then for some nonsingular matrix $S$ over $\FF$, 
	$X$ is spanned by the first $k$ columns of $tPS$ (modulo $L$). 
	Then $L \circ X = \langle PS(t^{{\bf 1}_{\leq k}}) \rangle_{\rm R}$ must hold.
	Indeed, $\supseteq$ is obvious, and the equality follows 
	from $\deg \langle PS(t^{{\bf 1}_{\leq k}}) \rangle_{\rm R} 
		= \deg L + k = \deg L + \dim X = \deg L \circ X$ (by (3)).
\end{proof}
\begin{Ex}
	Consider the simplest case of ${\cal L}_{\rm R}(\FF(t)^2)$ with $\FF = \ZZ/2\ZZ$.
    Then  $L \in {\cal L}_{\rm R}(\FF(t)^2)$ is spanned by two vectors $p_1,p_2 \in \FF(t)^2$;   
    we simply write it as  $L = \langle p_1, p_2 \rangle_{\rm R}$.
    According to the proof of Lemma~\ref{lem:[L,tL]},  
   $e_1 := tp_1 + L$ and $e_2 := tp_2 + L$ form an $\FF$-basis of $tL/L$.
	In particular, $tL/L$ is isomorphic to $\FF^2$ by  $\FF^2 \ni
	\left( \begin{array}{c}
	a_1 \\
	a_2
	\end{array} \right)
	\mapsto a_1 e_1 + a_2 e_2$.
	There are five subspaces in 
	$\FF^2$:  
	\begin{equation*}
	X_0 = \{0\},\ 
	X_1 = \FF \left( \begin{array}{c}
	1 \\
	0 
    \end{array} \right), \ 	
    X_2 = \FF \left( \begin{array}{c}
    0 \\
    1 
    \end{array} \right),  \
     X_3 = \FF \left( \begin{array}{c}
     1 \\
     1 
     \end{array} \right), \ X_4 = \FF^2.
	\end{equation*}	
	Then $[L, tL]$ consists of $L = L \circ X_0$, $t L = L \circ X_4$, and
	\begin{equation*}
	L \circ X_1 = \langle t p_1, p_2 \rangle_{\rm R},\   L \circ X_2 = \langle t p_2, p_1 \rangle_{\rm R},\  
	L \circ X_3 = \langle t (p_1+ p_2), p_2 \rangle_{\rm R}.
	\end{equation*}	
\end{Ex}

\subsection{L-convex function}
We first review L-convex functions on $\ZZ^n$; see~\cite[Chapter 7]{MurotaBook} for details.  
A function $g:\ZZ^n \to \RR \cup \{\infty\}$ is called {\em L-convex} if
it satisfies:
\begin{description}
	\item[{\rm (SUB$^{\ZZ}$)}] $g(x) + g(y) \geq g(\min(x,y)) + g(\max(x,y))$ for $x,y \in \ZZ^n$.
	\item[{\rm (LIN$^{+\mathbf{1}}$})] There is $r \in \RR$ such that
	$g(x+ {\bf 1}) = g(x) + r$  for $x \in \ZZ^n$.
\end{description}
We treat the infinity element $\infty$ as $\infty + c = \infty$ 
for $c \in \RR \cup \{\infty\}$ and $b < \infty$ for $b \in \RR$. 
Also we assume that any function $g: \ZZ^n \to \RR \cup \{\infty\}$ has a point $x$ with $g(x) < \infty$.  
\begin{Ex}\label{ex:L-convex_ex}
The following function $g:\ZZ^N \to \RR \cup \{\infty\}$ is known to be L-convex: 
\begin{equation*}
g(x) = \sum_{1 \leq i,j \leq N} \phi_{ij}(x_i - x_j)  \quad (x \in \ZZ^N),
\end{equation*}
where $\phi_{ij}:\ZZ \to \RR \cup \{\infty\}$ is a $1$-dimensional convex function for each $i,j$.
This L-convex function arises from the dual of minimum-cost network flow. 
\end{Ex}
We are interested in minimization of an L-convex function. 
Note that
$r = 0$ in (LIN$^{+\mathbf{1}})$ is a necessary condition 
for the existence of a minimizer.
We tacitly assume $r = 0$ in the sequel.
The following optimality property is basic. 
\begin{Lem}[\cite{Murota98_MPA}]\label{lem:opt}
	Let $g:\ZZ^n \to \RR \cup \{\infty\}$ be an L-convex function.
	A point $x \in \ZZ^n$ is a minimizer of $g$ if and only if
	\begin{equation}
	g(x) \leq g(x+ u) \quad (u \in  \{0,1\}^n).
	\end{equation} 
\end{Lem}
This property naturally leads to the following simple descent algorithm, 
called the {\em steepest descent algorithm}.
\begin{description}
	\item[Steepest Descent Algorithm (SDA$(\ZZ^n)$)]
	\item[Input:] An L-convex function $g:\ZZ^n \to \RR \cup \{\infty\}$.
    \item[Output:] A minimizer of $g$.
    \item[Step 0:] Choose $x^0 \in \ZZ^n$ with $g(x^0) < \infty$. Let $i := 0$. 
    \item[Step 1:] Find a minimizer $y$ of $g$ over $x^i + \{0,1\}^n$. 
    \item[Step 2:] If $g(y) < g(x^i)$, then let $i \leftarrow i+1$, $x^{i} \leftarrow y$, and go to step 1. 
	\item[Step 3:] Otherwise, output $x^i$ as a minimizer.
\end{description}
The point $y$ in Step 1 is called a {\em steepest direction} at $x^i$.
The function $u \mapsto  g(x + u)$ is submodular 
on Boolean lattice $\{0,1\}^n$. 
Hence a steepest direction can be found by 
a submodular function minimization on the Boolean lattice.
It is a fundamental fact in combinatorial optimization 
that any submodular function on the Boolean lattice
can be minimized in polynomial time; see e.g., \cite[Section 14.3]{KorteVygen}.
An intriguing property of SDA is the following bound of the number of the iterations.
\begin{Thm}[\cite{MurotaShioura14}]\label{thm:bound_ZZ^n}
	Let $g:\ZZ^n \to \RR \cup \{\infty\}$ be an L-convex function, 
	and let $k$ be the minimum $l_{\infty}$-distance between $x^0$ 
	and minimizers $y^* \geq x^0$, i.e.,
	\begin{equation}
	k := \min \{\|x^0 - y^* \|_{\infty} \mid \mbox{$y^*$ is a minimizer of $g$ with $y^* \geq x^0$ }   \}.
	\end{equation}
	In SDA, the $k$-th point $x^k$ is a minimizer of $g$.
\end{Thm}
Next we introduce L-convex functions on a uniform modular lattice, 
and generalize the above properties. 
The argument goes in a straightforward way.
Let ${\cal L}$ be a uniform modular lattice with uniform-rank $n$. 
A function $g:{\cal L} \to \RR \cup \{\infty\}$
is called {\em L-convex} if it satisfies:
\begin{description}
	\item[{\rm (SUB)}] $g(x) + g(y) \geq g(x \wedge y) + g(x \vee y)$ for all $x,y \in {\cal L}$.
	\item[{\rm (LIN$^+$)}] There is $\alpha \in \RR$ such that
	$g((x)^+) = g(x) + \alpha$  for all $x \in {\cal L}$.
\end{description}
Fix an arbitrary apartment system of ${\cal L}$.
Every apartment of ${\cal L}$ is a sublattice isomorphic to $\ZZ^n$ 
and preserves the ascending operation. 
Thus the L-convexity is characterized by the L-convexity on each apartment.
\begin{Lem}\label{lem:restriction}
A function $g: {\cal L} \to \RR \cup \{\infty\}$ is L-convex 
if and only if the restriction of $g$ to every apartment $\varSigma$
is L-convex, where $\varSigma$ is identified with $\ZZ^n$.
\end{Lem}
The optimality criterion (Lemma~\ref{lem:opt}) is generalized as follows. 
\begin{Lem}\label{lem:opt'}
	Let $g:{\cal L} \to \RR \cup \{\infty\}$ be an L-convex function.
	A point $x \in {\cal L}$ is a minimizer of $g$ if and only if
	\begin{equation}
	g(x) \leq g(y) \quad (y \in  [x, (x)^+]).
	\end{equation} 
\end{Lem}
\begin{proof}
	Suppose that $x$ is not a minimizer.
	Consider a minimizer $y^*$ of $g$.
	Choose an apartment $\varSigma$ containing $x$ and $y^*$.
	By $\varSigma \simeq \ZZ^n$ and Lemmas~\ref{lem:opt} and \ref{lem:restriction},
	there is $y \in [x,x+{\bf 1}] \subseteq [x,(x)^+]$ with $g(y) < g(x)$.
\end{proof}
The steepest descent algorithm is formulated as follows.
\begin{description}
	\item[Steepest Descent Algorithm (SDA$({\cal L})$)]
	\item[Input:] An L-convex function $g:{\cal L} \to \RR \cup \{\infty\}$.
	\item[Output:] A minimizer of $g$.
	\item[Step 0.] Choose $x^0 \in {\cal L}$ with $g(x^0) < \infty$. Let $i := 0$. 
	\item[Step 1.] Find a minimizer $y$ of $g$ over $[x^i, (x^i)^+]$. 
	\item[Step 2.] If $g(y) < g(x^i)$, then let $i \leftarrow i+1$, $x^{i} \leftarrow y$, and go to step 1. 
	\item[Step 3.] Otherwise, output $x^i$.
\end{description}
The interval $[x^i, (x^i)^+]$ is a complemented modular lattice, 
and $g$ is submodular on $[x^i, (x^i)^+]$.
In particular, Step 1 reduces to a submodular function minimization on the complemented modular lattice.
In the building-theoretic view,  
$[x^i, (x^i)^+]$ is the spherical building at the link of point $x^i$, 
and Step 1 is an optimization on the spherical building.

To generalize the iteration bound (Theorem~\ref{thm:bound_ZZ^n}), 
we introduce the $l_{\infty}$-distance on ${\cal L}$.  
For two elements $x,y \in {\cal L}$, 
choose an apartment $\varSigma$ containing $x,y$, identify $\varSigma$ 
with $\ZZ^n$, and define
the {\em $l_{\infty}$-distance} 
$d_{\infty}(x,y)$ by
\begin{equation}
d_{\infty}(x,y) := \|x - y\|_{\infty}. 
\end{equation}
One can see from the property (B2) in Lemma~\ref{lem:skeleton} that $d_{\infty}$ is 
independent of the choice of the apartment.
\begin{Thm}\label{thm:bound_L}
	Let $g:{\cal L} \to \RR \cup \{\infty\}$ be an L-convex function, 
	and let $k$ be the minimum $l_{\infty}$-distance between $x^0$ 
	and minimizers $y^* \succeq x^0$, i.e.,
	\begin{equation}
	k := \min \{ d_{\infty}(x^0,y^*) \mid \mbox{$y^*$ is a minimizer of $g$ with $y^* \succeq x^0$ }   \}.
	\end{equation}
	In SDA$({\cal L})$, the $k$-th point $x^k$ is a minimizer of $g$.
\end{Thm}
\begin{proof}
	We may suppose that a minimizer exists.
	It suffices to show that the distance $k$ decreases by $1$ on the update $x^0 \to x^1$. 
	Let $y^*$ be a minimizer of $g$ with $y^* \succeq x_0$ and $d_{\infty}(x^0,y^*) = k$.
	Choose apartment $\varSigma$ containing $y^*$ and short chain $\{ x^0,x^1\}$.
	Identify $\varSigma$ with $\ZZ^n$. Then $x^0$ is not a minimizer of 
	$g$ over $\varSigma$.
	Then the update $x^0 \to x^1$ is viewed as the update of 
	the first iteration of SDA$(\ZZ^n)$.
	Thus, by Theorem~\ref{thm:bound_ZZ^n}, the distance decreases on $\varSigma$ 
	and on ${\cal L}$.
\end{proof}

Let $\FF$ be a skew field, and let $\FF(t)$ be the skew field of rational functions. 
For (nonzero) $A \in \FF(t)^{n \times n}$, 
define $\deg A
: {\cal L}_{\rm L}(\FF(t)^n) \times {\cal L}_{\rm R}(\FF(t)^n) \to \ZZ$ by 
\begin{equation}
\deg A(L,M) := \max \{ \deg u \mid u \in A(L,M) \}. 
\end{equation} 
Notice that if $L = \langle P \rangle_{\rm L}$ and $M = \langle Q \rangle_{\rm R}$
then $\deg A(L,M)$ is equal to the maximum degree of an entry of $PAQ$.
An affine analogue of Lemma~\ref{lem:submo} is the following:
\begin{Lem}\label{lem:L-convexity}
	Let $A \in \FF(t)^{n \times n}$.
	Then the function $(L,M) \mapsto \infty \cdot \deg A(L,M)$ is L-convex on
	${\cal L}_{\rm L}(\FF(t)^n) \times \check{\cal L}_{\rm R}(\FF(t)^n)$, 
	where $\infty \cdot c$ is defined as $\infty$ 
	if $c > 0$ and $0$ if $c \leq 0$.
\end{Lem}
Here ${\cal L}_{\rm L}(\FF(t)^n) \times \check{\cal L}_{\rm R}(\FF(t)^n)$ 
is viewed as a uniform modular lattice with ascending operator 
$(L,M) \mapsto (tL,t^{-1}M)$.
\begin{proof}
	By Lemmas~\ref{lem:restriction} and \ref{lem:apartment}, 
	it suffices to show the L-convexity 
	on the apartment $\varSigma_{\rm L}(P) \times \check\varSigma_{\rm R}(Q)$ of ${\cal L}_{\rm L}(\FF(t)^n) \times \check{\cal L}_{\rm R}(\FF(t)^n)$.
	Suppose that the row vectors of $P$ are $p_1,p_2, \ldots, p_n$ and column vectors of $Q$ are $q_1,q_2, \ldots, q_n$. Then the apartment
	$\varSigma_{\rm L}(P) \times \check\varSigma_{\rm R}(Q)$ is isomorphic to 
	$\ZZ^{n} \times \ZZ^n$ by
	$(z,w) \mapsto (\langle (t^z) P \rangle_{\rm L}, \langle Q(t^{-w}) \rangle_{\rm R})$.
	Now $\deg A(\langle (t^z) P \rangle_{\rm L}, \langle Q(t^{-w}) \rangle_{\rm R}) = \max_{1 \leq i,j \leq n} \deg (t^{z_i} p_i  A t^{- w_i} q_j) 
	= \max_{1 \leq i,j \leq n} z_i - w_j + \deg p_i Aq_j$.
	Thus 
	\[
	\infty \cdot \deg A(\langle (t^z) P \rangle_{\rm L}, \langle Q(t^{-w}) \rangle_{\rm R}) = \sum_{1 \leq i,j \leq n}  \infty \cdot ( z_i - w_j + \deg p_i Aq_j).
	\]
	Notice that $x \mapsto \infty \cdot (x + b)$ is convex on $\ZZ$.
	By Example~\ref{ex:L-convex_ex} (with $N = 2n$),  
	 $(z,w) \mapsto \infty \cdot \deg A(\langle (t^z) P \rangle_{\rm L}, \langle Q(t^{-w}) \rangle_{\rm R})$
	 is L-convex on $\ZZ^n \times \ZZ^n = \ZZ^{2n}$.
	This means that $\infty \cdot \deg A$ is L-convex on 
	${\varSigma}_{\rm L}(P) \times \check{\varSigma}_{\rm R}(Q)$.
\end{proof}

\section{Computing the degree of determinants}\label{sec:computing}

The goal of this section is to establish a formula and algorithm for the degree 
of the Dieudonn\'e determinant of a linear symbolic matrix.
Let $A = A_0 + A_1 x_1 + \cdots + A_m x_m$ be a linear matrix over $\KK(t)$.
Now $A$ is viewed as a matrix over the skew field $\KK(\langle x_1,x_2,\ldots,x_m \rangle) (t)$ of rational functions over the free field $\KK(\langle x_1,x_2,\ldots,x_m \rangle)$.
As in the case of the nc-rank, 
we first give an upper bound of $\deg \Det A$.
The following upper bound of $\deg \Det A$ 
is observed by Murota~\cite{Murota95_SICOMP} for $\deg \det A$, 
which was a basis of the combinatorial relaxation algorithm.
\begin{Lem}\label{lem:bound_degDet}
For nonsingular matrices $P,Q$ over $\KK(t)$, 
if $P A_i Q$ is a proper matrix over $\KK(t)$ for $i=0,1,2,\ldots,m$, 
then $\deg \Det A \leq - \deg \det P - \deg \det Q$.
\end{Lem}
\begin{proof}
	$P A Q$ is a proper matrix over $\KK(\langle x \rangle) (t)$. 
	Also $\deg \det P = \deg \Det P$ and $\deg \det Q = \deg \Det Q$.
	Thus the claim follows from Lemmas~\ref{lem:DetAB} and \ref{lem:key}.
\end{proof}
This gives rise to the following optimization problem 
({\em maximum vanishing submodule problem (MVMP)\footnote{The second M means subModule.}}):
\begin{eqnarray*}
{\rm MVMP:} \quad {\rm Max.} && \deg \det P + \deg \det Q \\
{\rm s.t.} &&   \mbox{$P A_i Q$: proper} \quad (i =0,1,\ldots,m), \\
&& P,Q \in \KK(t)^{n \times n}: \mbox{nonsingular}.
\end{eqnarray*}
Just as 
MVSP is formulated as an optimization 
over the lattice of vector subspaces of $\KK^n$ (see (\ref{eqn:MVSP})), 
MVMP is reformulated as an optimization 
over the lattice of submodules of $\KK(t)^n$. 
Recall notions in Section~\ref{subsec:uniform}.
Then the above problem is rephrased as the following:
\begin{eqnarray*}
	{\rm MVMP:} \quad {\rm Max.} && \deg L + \deg M \\
	{\rm s.t.} &&   \deg A_i(L,M) \leq 0 \quad (i =0,1,\ldots,m), \\
	&& L \in {\cal L}_{\rm L}(\KK(t)^n),\ M \in {\cal L}_{\rm R}(\KK(t)^n).
\end{eqnarray*}
The following theorem states that this upper bound is tight for $\deg \Det$, which is an extension of 
the Fortin-Rautenauer formula (Theorem~\ref{thm:FortinReutenauer}).
The proof is given later.
\begin{Thm}\label{thm:degDet_formula}
	Let $A = A_0 + A_1 x_1 + A_2 x_2 + \cdots + A_m x_m$ be an $n \times n$ linear matrix over $\KK(t)$. 
	Then $\deg \Det A$ is equal to the negative of 
	the optimal value of MVMP.
\end{Thm}
Then $\deg \Det$ is an upper bound of $\deg \det$, analogously to 
the relation between $\rank$ and $\ncrank$.
\begin{Cor}\label{cor:deg_leq_Deg}
	$\deg \det A \leq \deg \Det A$.
\end{Cor}
\begin{proof}
	Let $(\langle P \rangle_{\rm L}, \langle Q \rangle_{\rm R})$ be 
	an optimal solution for MVMP. By Theorem~\ref{thm:degDet_formula}, 
	it holds $\deg \Det A = - \deg \det P - \deg \det Q$.
	Now $P A Q$ is also a proper matrix over $\KK(x)(t)$. 
	By Lemma~\ref{lem:key} 
	we have $\deg \det A \leq - \deg \det P - \deg \det Q = \deg \Det A$.
\end{proof}
We also give an algorithm to solve MVMP for a polynomial matrix $A$.
\begin{Thm}\label{thm:degDet_algo}
	Let $A = A_0 + A_1 x_1 + A_2 x_2 + \cdots + A_m x_m$ be an $n \times n$ linear matrix over $\KK[t]$. 
	MVMP can be solved in $O(\ell n \gamma + \ell^2 m n^{\omega+2})$ time, and 
	 in $O((\ell- \alpha_n) \gamma + (\ell - \alpha_n)^2 m n^{\omega})$ time
		if $A$ is nonsingular, where
		\begin{itemize}
			\item $\gamma$ is the time complexity of solving MVSP for an $n \times n$ linear matrix over $\KK$, 
			\item $\ell (= \alpha_1)$ is the maximum degree of entries in $A$, 
			\item $\alpha_n$ is the minimum degree of the Smith-McMillan form of $A$ in $\KK(\langle x\rangle)(t)$, and
			\item $\omega$ is the exponent of the time complexity 
			of matrix multiplication of $n \times n$ matrices. 
		\end{itemize}
\end{Thm}
Here we make a strong assumption that arithmetic operations on $\KK$ 
can be done in constant time.  
The bit-length consideration of our algorithm for the case $\KK = \QQ$ is left to future work.
On this direction, 
Oki~\cite{OkiHJ19} devised an alternative algorithm for 
$\deg \Det$, which works with a bounded bit-length. 
Theorems~\ref{thm:degDet_formula} and \ref{thm:degDet_algo}  are proved in the subsequent subsections.
\begin{Rem}
	The maximum degree of the subdeterminants of $n \times n'$ linear matrix $A$
	can be computed by combining the above result with 
	the valuated-matroid property (Proposition~\ref{prop:valuated}).
	Indeed, consider the expanded matrix $\tilde A := (I\ A)$ 
	and the valuated matroid $\omega$ obtained from column vectors of $\tilde A$.
	Then the maximum degree of the subdeterminants of $A$ 
	is equal to the maximum value of $\omega(X)$ 
	over $X \subseteq \{1,2,\ldots,n+n'\}$ with $|X| = n$.
	By the greedy algorithm~\cite{DressWenzel_greedy} 
	(see also \cite[Section 5.2.4]{MurotaMatrix}), it is obtained 
	by $O((n+n')^2)$ evaluations of $\omega$, 
	where the evaluation is done by the algorithm in Theorem~\ref{thm:degDet_algo}.
\end{Rem}

\subsection{Optimality}
Here we establish an optimality 
criterion for MVMP, and prove Theorem~\ref{thm:degDet_formula}.
We first note that MVMP can be viewed as L-convex function minimization 
on a uniform modular lattice:
\begin{eqnarray}
{\rm Min.} && - \deg L - \deg M + \sum_{i=0}^m \infty \cdot \deg A_i(L,M) \nonumber \\ 
{\rm s.t.} && (L,M) \in {\cal L}_{\rm L}(\KK(t)^n) \times \check{\cal L}_{\rm R}(\KK(t)^n). \label{eqn:L-convex-min}
\end{eqnarray}
Recall Lemma~\ref{lem:L-convexity} for the notation $\infty \cdot \deg A_i(L,M)$.
Then the objective function is actually L-convex. 
Indeed, by Lemma~\ref{lem:L-convexity}, 
the functions in the summation are L-convex. 
Recall Lemma~\ref{lem:L(F(t)^n)} that  
$\deg$ is a unit valuation on a uniform modular lattice.
Then $(L,M) \mapsto - \deg L - \deg M$ is L-convex on 
${\cal L}_{\rm L}(\KK(t)^n) \times \check{\cal L}_{\rm R}(\KK(t)^n)$ 
with $- \deg t L - \deg t^{-1} M = - \deg L - \deg M$.
Notice from the definition that the sum of L-convex functions is L-convex.

This fact and Lemma~\ref{lem:opt'} motivate us to consider the restriction of MVMP to
interval $[(L,M),(L,M)^+] = [(L,M),(tL,t^{-1}M)]$ for 
a feasible solution $(L, M)$ of MVMP.
Since $\deg A_i (L,M) \leq 0$, it holds $\deg A_i (L',M') \leq 1$ for $(L',M') \in [(L,M),(L,M)^+]$; see Lemma~\ref{lem:[L,tL]}~(4).
To study the feasibility of MVMP on $[(L,M),(L,M)^+]$,
we may consider the coefficient of $t$ in $A_i(tu,v)$ for 
$u \in L, v \in M$.
For each $A_i$, 
define a bilinear map $A_i^{L,M}: tL/L \times M/t^{-1}M  \to \KK$ 
by
\begin{eqnarray*}
A_i^{L,M}(tu+L, v + t^{-1}M) & := & A_i(u,v)^0  \\ 
& = & \mbox{the coefficient of $t$ in $A_i(tu,v)$ }  \quad (u \in L, v \in M).
\end{eqnarray*}
This is well-defined (by $A(L,M) \subseteq \KK(t)^-$).
Define MVSP$^{L,M}$ by
\begin{eqnarray*}
{\rm MVSP}^{L,M}: \quad  {\rm Max.} && \dim X + \dim Y  \\
{\rm s.t.} && A^{L,M}_i(X,Y)  = \{0\} \quad (i=0,1,\ldots,m), \\
&& X \in {\cal S}(tL/L), Y \in {\cal S}(M/t^{-1}M).
\end{eqnarray*}
Recall Lemma~\ref{lem:[L,tL]} for notation $L \circ X$ for $X \in {\cal S}(tL/L)$.
Also, for $Y \in {\cal S}(M/t^{-1}M)$, define $M \bullet Y := t^{-1} M \circ Y$.
Then the following lemma verifies 
that MVSP$^{L,M}$ is the restriction of MVMP to 
$[(L,M),(L,M)^+]  (\simeq {\cal S}(tL/L) \times \check{\cal S}(M/t^{-1} M))$:
\begin{Prop}\label{prop:augment_MVMP}
	Let $(L,M)$ be a feasible solution of MVMP.
	For $(X,Y) \in {\cal S}(tL/L) \times {\cal S}(M/t^{-1}M)$, we have the following:
	\begin{itemize}
		\item[{\rm (1)}] $\deg L \circ X + \deg M \bullet Y = \deg L + \deg M + (\dim X + \dim Y - n)$.
		\item[{\rm (2)}] $(L \circ X, M \bullet Y)$ is feasible to MVMP if and only if 
		$(X,Y)$ is feasible to MVSP$^{L,M}$.
	\end{itemize}
\end{Prop}
\begin{proof}
	(1) follows from Lemma~\ref{lem:[L,tL]} (3).
	(2) follows from:
	\begin{eqnarray*}
	&& \deg A_i (L \circ X, M \bullet Y) \leq 0 \\
	&\Leftrightarrow& \deg A_i (tu, v) \leq 0 \quad (u \in L: tu+ L \in X,\ v \in M: v+ t^{-1}M \in Y ) \\
	&\Leftrightarrow& A_i (u, v)^0 = 0 \quad (u \in L: tu+ L \in X,\ v \in M: v+ t^{-1}M \in Y ) \\
	&\Leftrightarrow& A_i^{L,M}(X,Y) = \{0\}.
	\end{eqnarray*}
\end{proof}
%
Suppose that $L$ and $M$ are given as 
$L = \langle P \rangle_{\rm L}$ and $M = \langle Q \rangle_{\rm R}$.
Then MVSP$^{L,M}$ is also written as
\begin{eqnarray*}
	{\rm MVSP}^{P,Q}: \quad  {\rm Max.} && \dim X + \dim Y \\
	{\rm s.t.} && (PA_iQ)^0 (X,Y)  = \{0\} \quad (i=0,1,\ldots,m), \\
	&& X \in {\cal S}(\KK^n), Y \in {\cal S}(\KK^n).
\end{eqnarray*}
Now we have the following optimality criterion.
\begin{Prop}\label{prop:opt_MVMP}
	For a feasible solution $(L,M) = ( \langle P \rangle_{\rm L}, \langle Q \rangle_{\rm R})$ of MVMP, 
	the following conditions are equivalent:
\begin{itemize}
	\item[{\rm (1)}] $(L,M)$ is optimal to MVMP.
	\item[{\rm (2)}] The optimal value of MVSP$^{L,M}$ is at most $n$ (is equal to $n$).
	\item[{\rm (3)}] The linear matrix
	\begin{equation*}
	(PAQ)^0 = (PA_0 Q)^0 + (PA_1 Q)^0 x_1 + \cdots + (PA_m Q)^0 x_m
	\end{equation*}
	is nonsingular on $\KK(\langle x \rangle)$.
	\item[{\rm (4)}] $\deg \Det A$ is equal to $- \deg \det P- \deg \det Q$.
\end{itemize}
\end{Prop}
\begin{proof}
	(4) $\Rightarrow$ (1) follows from Lemma~\ref{lem:bound_degDet}.
	(2) $\Rightarrow$ (3) follows from Theorem~\ref{thm:FortinReutenauer}.
	Indeed, the optimal value of MVSP$^{L,M}$ is equal to $2n - \ncrank (PAQ)^0$.
	If the value is at most $n$, then $\ncrank (PAQ)^0 \geq n$, 
	and hence $\ncrank (PAQ)^0 = n$, i.e., $(PAQ)^0$ is nonsingular on $\KK(\langle x \rangle)$. 
	(3) $\Rightarrow$ (4) follows from Lemma~\ref{lem:key}. 
	(1) $\Rightarrow$ (2) follows from Proposition~\ref{prop:augment_MVMP}.
\end{proof}


\begin{proof}[Proof of Theorem~\ref{thm:degDet_formula}]
	We may assume that MVMP is bounded.
	Take a feasible solution $(L,M) = (\langle P \rangle_{\rm L}, \langle Q \rangle_{\rm R})$.  If $(PAQ)^0$ is nonsingular, 
	then $\deg \Det A = - \deg L - \deg M$. 
	Otherwise, we obtain another feasible solution $(L',M')$ of MVMP 
	with  $\deg L' + \deg M' > \deg L + \deg M$. Let $(L,M) \leftarrow (L',M')$.
	Repeating this procedure finitely many times, 
	we obtain $\deg \Det A = - \deg L - \deg M$.
	
	Notice that the proof is also obtained directly
	from Proposition~\ref{prop:opt_MVMP}~(1)~$\Leftrightarrow$~(4).
\end{proof}

\subsection{Steepest descent algorithm}
The above proof of Theorem~\ref{thm:degDet_formula} is algorithmic, 
and naturally leads to the following algorithm, which can be viewed as the steepest descent algorithm for the L-convex function in (\ref{eqn:L-convex-min}). 
\begin{description}
	\item[Steepest Descent Algorithm for $\deg \Det$ (coordinate-free version)]
	\item[Input:] A linear matrix $A = A_0 + A_1 x_1 + \cdots + A_m x_m$ over $\KK(t)$.
	\item[Output:] The degree $\deg \Det A$ of the Dieudonn\'e determinant of $A$. 
	\item[Step 0:] Choose a feasible solution $(L,M)$ of MVMP.
	\item[Step 1:] Solve MVSP$^{L,M}$ to obtain an mv-subspace $(X,Y)$. 
	\item[Step 2:] If $\dim X + \dim Y \leq n$, then $(L,M)$ is optimal to MVMP 
	and output $- \deg L - \deg M$.
	\item[Step 3:] Let $(L,M) \leftarrow (L \circ X, M \bullet Y)$, and go to step 1. 	
\end{description}
In step 1, the algorithm chooses an mv-subspace $(X,Y)$, 
and therefore $(L \circ X, M \bullet Y)$ is actually a steepest direction at $(L,M)$.
The input is allowed to be a linear rational matrix $A$.
If $A$ is nonsingular, then the algorithm 
outputs the correct answer after finitely many iterations; the exact number of iterations will be given in Lemma~\ref{lem:iterations}.
In the case of singular $A$,   
we do not know when the algorithm should output $-\infty$. 

We next consider the case of a linear polynomial matrix, 
and prove Theorem~\ref{thm:degDet_algo}. 
We specialize the above algorithm with a matrix form.
\begin{description}
	\item[Steepest Descent Algorithm for $\deg \Det$ (matrix version)]
	\item[Input:] A linear matrix $A = A_0 + A_1 x_1 + \cdots + A_m x_m$ over $\KK[t]$, 
	where $\ell$ is the maximum degree of entries of $A$.
	\item[Output:] The degree $\deg \Det A$ of the Dieudonn\'e determinant of $A$. 
	\item[Step 0:]  Let $A_i \leftarrow A_i t^{-\ell}$ for $i=0,1,2,\ldots,m$, and 
	$D^* \leftarrow  n\ell$.	
	\item[Step 1:] Solve MVSP in the matrix form
\begin{eqnarray*}
	{\rm Max.} && r + s  \\
	{\rm s.t.} && \mbox{$S A^{0}_i T$ has a zero submatrix in first $r$ rows and first $s$ columns,}\\
	&& S, T \in \KK^{n \times n}: \mbox{nonsingular},
\end{eqnarray*}
	and obtain optimal matrices $S,T$.
	\item[Step 2:] If the optimal value $r+s$ is at most $n$, 
	then output $D^*$.
	\item[Step 3:] Let $A_i \leftarrow (t^{{\bf 1}_{\leq r}})S A_i T 
	(t^{- {\bf 1}_{> s}} )$ for $i=0,1,2,\ldots,m$, and $D^* \leftarrow D^* - (r+s - n)$.
	If $D^* < 0$, then output $- \infty$.
	Go to step 1 otherwise.
\end{description}
Here ${\bf 1}_{>s} := {\bf 1} - {\bf 1}_{\leq s}$.
Notice that the matrix version changes the input linear matrix $A$ in each iteration.
In step 0, we suppose feasible module 
$(L,M) = (\langle P \rangle_{\rm L}, \langle Q \rangle_{\rm R})= (I, t^{-\ell} I)$　
with $\deg L + \deg M = - D^* = - n \ell$.
The update in step 3 can be understood as 
the movement from $(L,M) = (\langle P \rangle_{\rm L}, \langle Q \rangle_{\rm R})$ 
to a steepest direction $(\langle (t^{{\bf 1}_{\leq r}}) SP \rangle_{\rm L}, \langle QT(t^{- {\bf 1}_{> s}} ) \rangle_{\rm R})$ in $[(L,M), (tL,t^{-1}M)]$; see Lemma~\ref{lem:[L,tL]}~(4).
Since the input is a polynomial matrix, $\deg \Det A$ is guaranteed 
to be nonnegative if $A$ is nonsingular (Lemma~\ref{lem:bound_degDet}).
Also $D^*$ is always an upper bound of $\deg \Det A$.
Thus $D^* < 0$ in step 3 implies $\deg \Det A = - \infty$.
 
The following modification of step 3 is natural.
\begin{description}
	\item[Step 3$'$:] Choose the minimum integer $\kappa \geq 1$ 
	such that $((t^{\kappa{\bf 1}_{\leq r}})S A T 
	(t^{- \kappa {\bf 1}_{>s}}))^0$
	has a nonzero submatrix in first $r$ rows and $s$ columns.
	Let $A_i \leftarrow (t^{\kappa{\bf 1}_{\leq r}})S A_i T 
	(t^{- \kappa {\bf 1}_{>s}} )$ for $i=0,1,2,\ldots,m$, and let $D^* := D^* - \kappa(r+s - n)$.
	If $\kappa$ is unbounded or $D^* < 0$, then output $\deg \Det A = -\infty$. 
	Go to step 1 otherwise.
\end{description}
The coordinate-free formulation cannot incorporate
this modification, since it depends on basis matrices for the current 
$(L,M)$ and the mv-subspace in step~2. 
The modified SDA using step 3$'$ is considered in Section~\ref{subsec:classical}.

We next estimate the number of iterations 
by using L-convexity (Theorem~\ref{thm:bound_L}).  
For this purpose, we consider the master problem $\overline{\rm MVMP}$ of MVMP:
\begin{eqnarray*}
	\overline{\rm MVMP}: \quad {\rm Max.} && \deg L + \deg M \\
	{\rm s.t.} &&   \deg A(L,M) \leq 0, \\
	&& L \in {\cal L}_{\rm L}(\KK(\langle x \rangle)(t)^n),\ M \in {\cal L}_{\rm R}(\KK(\langle x \rangle)(t)^n),
\end{eqnarray*}
where the linear matrix $A$ is regarded as a bilinear form 
on $\KK(\langle x \rangle)(t)^n \times \KK(\langle x \rangle)(t)^n$.
Solving $\overline{\rm MVMP}$ is theoretically easy.
Choose biproper matrices $P,Q$ so that $PAQ$ is the Smith-McMillan form. 
Now $PAQ$ is the diagonal matrix $(t^{\alpha})$ for $\alpha \in \ZZ$.
Consider $L^* := \langle (t^{- \alpha^-})P \rangle_{\rm L}$ and 
$M^* := \langle Q(t^{- \alpha^+}) \rangle_{\rm R}$ for 
$\alpha^+ := \max ({\bf 0},  \alpha)$ and $\alpha^- := \min ( {\bf 0}, \alpha)$.
Then $(L^*,M^*)$ is feasible to $\overline{\rm MVMP}$.
Also $\deg L^* + \deg M^* = - \sum_{i=1}^n \alpha_i = - \deg \Det A$, and
hence $(L^*,M^*)$ is an optimal solution.

As for MVSP embedded to $\overline{\rm MVSP}$,
MVMP is embedded to $\overline{\rm MVMP}$ 
by the scalar extension $(L, M) \mapsto (\KK(\langle x \rangle)(t)^- \otimes L, M \otimes \KK(\langle x \rangle)(t)^-)$.
In particular MVMP
is an exact inner approximation of $\overline{\rm MVMP}$.
We further show that 
the steepest descent algorithm for MVMP 
is viewed as that for $\overline{\rm MVMP}$.
Let $(L,M) = (\langle P \rangle_{\rm L}, \langle Q \rangle_{\rm R})$ 
be a feasible solution of MVMP and of $\overline{\rm MVMP}$.
Consider MVSP$^{P,Q}$, and then $\overline{\rm MVSP}^{P,Q}$, 
which is given by
\begin{eqnarray*}
	\overline{\rm MVSP}^{P,Q}: \quad  {\rm Max.} && \dim X + \dim Y \\
	{\rm s.t.} && (PAQ)^0 (X,Y)  = \{0\}, \\
	&& X \in {\cal S}_{\rm L}(\KK(\langle x \rangle)^n),  
	Y \in {\cal S}_{\rm R}(\KK(\langle x \rangle)^n).
\end{eqnarray*}
By Lemma~\ref{lem:innerapprox}, 
any mv-subspace of MVSP$^{P,Q}$ is also an mv-subspace of 
$\overline{\rm MVSP}^{P,Q} = \overline{{\rm MVSP}^{P,Q}}$.
Thus we have:
\begin{Lem}
	A steepest direction for MVMP at $(L,M)$
	is also a steepest direction for $\overline{\mbox{MVMP}}$ at $(L,M)$.
\end{Lem}
We next show the exact number of the iterations of SDA, where by the number of the iterations we mean the number of the updates of $(L,M)$ (or $A$).	
\begin{Lem}\label{lem:iterations}
	If $A$ is nonsingular, then the number of the iterations of 
	the steepest descent algorithm is 
	equal to $\alpha_1 - \alpha_n$, 
	where $\alpha_1$ and $\alpha_n$ are the maximum and 
	minimum degrees, respectively, of the Smith-McMillan form of $A$.
\end{Lem}
\begin{proof}
	Notice that $\ell = \alpha_1$.
	By the initial update $A \leftarrow At^{-\alpha_1}$, 
	we can assume that $\alpha_1 = 0 \geq \alpha_n$ and 
	the initial point $(L,M)$ is 
	$(\langle I \rangle_{\rm L}, \langle I \rangle_{\rm R})$.
	An optimal solution $(L^*,M^*)$ of $\overline{\rm MVMP}$ 
	with $(L^*,M^*) \succeq (L,M)$ is given 
	by $(\langle (t^{-\alpha}) P \rangle_{\rm L}, \langle Q \rangle_{\rm R})$ 
	for $\alpha = (\alpha_1,\alpha_2,\ldots,\alpha_n)$ and biproper $P,Q$.
	By $(\langle I \rangle_{\rm L}, \langle I \rangle_{\rm R}) =  (\langle P \rangle_{\rm L}, \langle Q \rangle_{\rm R})$,  
	both the initial point and the optimum belong to 
	the apartment $\varSigma_{\rm L}(P) \times \check \varSigma_{\rm R}(Q)$
	in ${\cal L}_{\rm L}(\KK(\langle x \rangle)(t)^n) \times  \check{\cal L}_{\rm R}(\KK(\langle x \rangle)(t)^n)$.
	Hence the $\ell_{\infty}$-distance from initial point to optimal solutions
	is at most $-\alpha_n$. 
	By Theorem~\ref{thm:bound_L}, the number of iteration is at most $\alpha_n$.
	The algorithm terminates when $\alpha_n=0$, i.e., $A$ becomes biproper (Lemma~\ref{lem:key}).	
    Thus it suffices to show that 
   $\alpha_n (< 0)$ increases by at most one 
   in the update $A \leftarrow (t^{{\bf 1}_{\leq r}})S A T 
   (t^{- {\bf 1}_{> s}})$. 
	Obviously
	$\delta_n = \deg \Det A$ increases by $r+s-n$.
	Also $\delta_{n-1}$ increases by $r+s - n$, $r+s -1 -n$, or $r+s + 1- n$; 
	then $\alpha_n = \delta_n - \delta_{n-1}$ increases by most one, as required. 
	
	The increase of $\delta_{n-1}$ can be seen as follows.
	Notice first that the update $A \leftarrow SAT$ does not change $\delta_{n-1}$.
    In the next update $A  \leftarrow (t^{{\bf 1}_{\leq r}}) A 
    (t^{- {\bf 1}_{> s}})$, the degree of an $(n-1) \times (n - 1)$ submatrix of $A$ 
    increases by $r+s +1 - n$ if the submatrix has all of the first $r$ rows and misses one of the last $n - s$ columns, by $r+s -1 -n$ if the submatrix misses one of the first $r$ rows and has all of the last $n - s$ columns, and by $r+s -n$ otherwise.
\end{proof}

\begin{proof}[Proof of Theorem~\ref{thm:degDet_algo}]
	We verify the time complexity of SDA (matrix form).
	After the initialization (step 0), each matrix $A_i$ is kept in the form
	\begin{equation}\label{eqn:expression}
		A_i^0 + A_i^{(1)}t^{-1} + \cdots + A_i^{(d)} t^{-d}, 
	\end{equation}
	where  $A_i^{(j)}$ is a matrix over $\KK$, and $d := \ell$.
	Step 1 can be done in $\gamma$ time.
	The update of expression (\ref{eqn:expression}) 
	in Step 3 can be done in $O(d mn^{\omega})$ time.
	The total number of iterations is $n\ell$ if $A$ is singular, and $\ell - \alpha_n$ if $A$ is nonsingular (Lemma~\ref{lem:iterations}).
	In each iteration, $d$ increases by one.
	Thus the total is $O(\ell n \gamma + \ell^2 mn^{\omega+2})$ 
	time if $A$ is singular, and is $O( (\ell - \alpha_n) \gamma +  (\ell - \alpha_n)^2 mn^{\omega})$ if $A$ is nonsingular. 
\end{proof}

\subsection{Combinatorial relaxation algorithm}
The steepest descent algorithm changes 
basis matrices $P,Q$ in each iteration.
It is a natural idea to optimize
on the apartment $\varSigma_{\rm L} (P) \times \varSigma_{\rm R}(Q)$ 
in each iteration.
This modification can be expected to reduce matrix operations, and
leads to the following algorithm, 
which is viewed as a generalization of the {\em combinatorial relaxation algorithm} previously developed for $\deg \det$~\cite{IwataOkiTakamatsu17,IwataTakamatsu13,Murota90_SICOMP,Murota95_SICOMP}. 
\begin{description}
	\item[Combinatorial Relaxation Algorithm for $\deg \Det$]
	\item[Input:] A linear matrix $A = A_0 + A_1 x_1 + \cdots + A_m x_m$ over $\KK[t]$, where $\ell$ is the maximum degree of entries of $A$.
	\item[Output:] The degree $\deg \Det A$ of the Dieudonn\'e determinant of $A$ 
	\item[Step 0:] Let $A_i \leftarrow A_i t^{-\ell}$ for $i=0,1,2,\ldots,m$, and 
	$D^* \leftarrow  n\ell$.
	\item[Step 1:] If $A^0$ is nonsingular, then output $D^* = \deg \Det A$. 
	
	\item[Step 2:] Find nonsingular matrices $S,T \in \KK^{n \times n}$ such that
	each $S A^{0}_i T$ $(i=0,1,2,\ldots,m)$ has a zero submatrix in first $r$ rows 
	and first $s$ columns with $r+s > n$. 
	\item[Step 3:] Solve the following problem:
	\begin{eqnarray*}
	{\rm MVMP}(\varSigma): \quad {\rm Max.} &&  \sum_{i} p_i - \sum_{i} q_i  \\
	{\rm s.t.} && \mbox{$(t^{p}) S A T (t^{-q})$ is proper},\\
	&& p,q \in \ZZ^n_+
	\end{eqnarray*}
	to obtain optimal vectors $p,q \in \ZZ^n_+$. 
	Let $A_i \leftarrow (t^{p})S A_i T (t^{- q})$ for $i=0,1,2,\ldots,m$, 
	and let $D^* \leftarrow D^* - \sum_{i} p_i + \sum_{i} q_i$.
	If $D^* < 0$ or {\rm MVMP}$(\varSigma)$ is unbounded, 
	then output $\deg \Det A := -\infty$.
	Otherwise, go to step 1.
\end{description}
The condition $r+s > n$ in step 1 guarantees that $D^*$ strictly decreases. 
Hence the algorithm terminates after $\ell n$ steps.
If the current solution $(L,M)$ is given by $(\langle P \rangle_{\rm L}, \langle Q \rangle_{\rm R})$, then 
MVMP$(\varSigma)$ is viewed as the restriction of MVMP 
to the apartment $\varSigma_{\rm L}(SP) \times \varSigma_{\rm R}(QT)$. 
Moreover, MVMP$(\varSigma)$ is 
the dual of the weighted matching problem in a bipartite graph. 
Indeed, the condition that $(t^{p})S A T(t^{-q})$ is proper
is written as 
\begin{equation}
p_i - q_j + d_{ij} \leq 0 \quad (1 \leq i,j \leq n), 
\end{equation}
where $d_{ij}(\leq 0)$ is the maximum degree of the $(i,j)$-entry of $SAT$.
Thus MVMP$(\varSigma)$ is the dual of the following weighted perfect matching problem:
\begin{eqnarray*}
	{\rm Max.} && \sum_{i=1}^n  d_{i \sigma(i)} \\
	{\rm s.t.} && \sigma:\mbox{permutation on $\{1,2,\ldots,n\}$}, 
\end{eqnarray*}
which can be efficiently solved by the Hungarian method 
to obtain optimal solution $p,q$ of the dual.

The combinatorial relaxation algorithm is seemingly more efficient
than the steepest descent algorithm, although 
we do not know any nontrivial iteration bound.  
The meaning of ``relaxation" is explained as follows.
Step 3 can be viewed as a relaxation process 
that the linear matrix $A$ is ``relaxed"  
into another linear matrix $\tilde{A}$ by replacing
each leading term $a_{ij} t^{d_{ij}}$ $(a_{ij} \in \KK)$ of $A$ with  
$x_{ij} t^{d_{ij}}$ for a new variable $x_{ij}$.
The optimal value of MVMP$(\varSigma)$ 
is the negative of $\deg \Det \tilde A$, and $\deg \Det A \leq \deg \Det \tilde A$.
Step 1 tests whether the relaxation is tight or not.  

\section{Linear symbolic matrix with rank-$1$ summands}\label{sec:rank-1}
In this section, we study a class of linear matrices 
$A = A_0 + A_1 x_1 + \cdots + A_m x_m$ 
for which $\deg \det A = \deg \Det A$ holds.
In the case of (nc-)rank,
Lov\'asz~\cite{Lovasz89} showed 
that if each summand $A_i$ is a rank-1 matrix, 
then the rank of $A$ is given by MVSP, i.e., $\rank A = \ncrank A$.
Ivanyos, Karpinski, and Saxena~\cite{IKS10} extended 
this result to the case where each $A_i$ other than $A_0$
has rank one.  
\begin{Thm}[\cite{IKS10}]\label{thm:rank-1}
	Let $A = A_0 + A_1 x_1 + \cdots + A_m x_m$ be a linear matrix over field $\KK$. 
	If $A_1,A_2,\ldots,A_m$ are rank-$1$ matrices, then
	$\rank A = \ncrank A$.
\end{Thm}
We remark that the rank computation of such a matrix reduces to linear matroid intersection~\cite{Lovasz89,Soma14}.

We show that Theorem~\ref{thm:rank-1} is naturally extended 
to the degree of the determinant
of linear matrix $A = A_0 + A_1 x_1 + \cdots + A_m x_m$ over $\KK(t)$.
\begin{Thm}\label{thm:rank-1_degdet}
	Let $A = A_0 + A_1 x_1 + \cdots + A_m x_m$ be a linear matrix over $\KK(t)$.
	If $A_1,A_2,\ldots,A_m$ are rank-$1$ matrices, 
	then $\deg \det A = \deg \Det A$.
\end{Thm}
\begin{proof}
	Consider an optimal module 
	$(\langle P \rangle_{\rm L}, \langle Q \rangle_{\rm R})$ for MVMP.
	By Proposition~\ref{prop:opt_MVMP}, 
	the linear matrix $(PAQ)^0$ is nonsingular as a matrix over $\KK(\langle x \rangle)$.
	Notice that each $(PA_iQ)^0$ for $i=1,2,\ldots,m$ has rank one.
	Indeed, $PA_iQ$ is a rank-1 matrix over $\KK(t)$, and hence is written as $\tilde u \tilde v^{\top}$ 
	for (nonzero) $\tilde u, \tilde v \in \KK(t)^n$.
	Consider the maximum degrees $c$ and $d$ of the components of $\tilde u$ and $\tilde v$, respectively. Necessarily $c+d \leq 0$ (since $PA_iQ$ is proper), and
	$PA_iQ$ is written as $t^{c}u t^{d}v^{\top}$ for $u,v \in (\KK(t)^-)^n$. 
    Then $(PA_iQ)^0 = u^0(v^0)^\top$ if $c+d = 0$ and zero if $c+d < 0$.
    
	By Theorem~\ref{thm:rank-1}, the linear matrix $(PAQ)^0$ is also nonsingular 
	as a matrix over $\KK(x)$.
	By Lemma~\ref{lem:key}, 
	we have $\deg \det A = - \deg \det P - \deg \det Q = \deg \Det A$.
\end{proof}

\begin{Rem}
	Observe from the Fortin-Rautenauer formula (Theorem~\ref{thm:FortinReutenauer}) that the nc-$\rank A$ 
	is a property of the matrix vector subspace 
	${\cal A} \subseteq \KK^{n \times n'}$ 
	spanned by $A_0, A_1,\ldots,A_m$ over $\KK$.
    Therefore,  $\rank A= \ncrank A$ still holds 
    if  the matrix subspace ${\cal A}' \subseteq \KK^{n \times n'}$
    spanned by  $A_1,\ldots,A_m$ admits a rank-1 basis $B_1, B_2,\ldots, B_{m'}$.
    Indeed, the constraint $A_i(X,Y) = \{0\}$ in MVSP can be replaced 
    by $B_i(X,Y) = \{0\}$. 
    See \cite{Gurvits04,IvanyosKarpinskiOiaoSantha15} for 
    the rank computation of such a linear matrix with a hidden rank-1 basis.
    
    In the case of $\deg \Det$, instead of the matrix vector space, 
    we may consider the matrix submodule ${\cal A} \subseteq \KK(t)^{n \times n}$ generated by $A_1,A_2,\ldots,A_m$ over $\KK(t)^-$.
    Since $\KK(t)^-$ is a PID and 
    ${\cal A}$ is a submodule of a free module 
    of matrices with bounded degree entries, 
    ${\cal A}$ is also free, and has a $\KK(t)^-$-basis.
   Analogously to the above, 
   $\deg \det A = \deg \Det A$ holds 
   if  the matrix module ${\cal A}'$
   generated by  $A_1,\ldots,A_m$ admits a rank-1 basis $B_1, B_2,\ldots, B_{m'}$; the constraint $\deg A_i(L,M) \leq 0$ can be replaced 
   by $\deg B_i(L,M) \leq 0$.
\end{Rem}

\begin{Rem}
	An important example of a linear matrix
	with possibly $\rank  < \ncrank$
	is a skew-symmetric linear matrix  $A = \sum_{i=1}^m A_i x_i$ 
	with rank-2 skew-symmetric summands~$A_i$; see the next example.
	The problem of 
	computing the (usual) rank of such a matrix is 
	a generalization of the nonbipartite matching problem, and is     
	equivalent to the {\em linear matroid parity problem}; see~\cite{Lovasz89}.
	Recently Iwata and Kobayashi~\cite{IwataKobayashi17} developed a polynomial time algorithm
	for the {\em weighted} linear matroid parity problem 
	by considering $\deg \det$ and using the idea of the combinatorial relaxation method.
	It is an interesting future direction to refine our non-commutative framework for   
	skew-symmetric linear matrices to capture
    nonbipartite matching and its generalizations.
\end{Rem}

\begin{Ex}
		Consider the following $3$ by $3$ linear skew-symmetric matrix $A = A_1x_1 + A_2 x_2 + A_3x_3$ (with rank-2 summands):  
		\[
		A =  \left(  
		\begin{array}{ccc}
		0 & x_1 & x_2 \\
		-x_1 & 0 & x_3 \\
		-x_2 & -x_3 & 0
		\end{array} \right).
		\]
		Then it is obvious that $\rank A = 2$. 
		However $\ncrank A = 3$. Indeed, it holds that $(1\ u \ v)A_i(1\ u'\ v')^{\top} = 0$ $(i=1,2,3)$ implies 
		$u= u'$ and $v = v'$ and $(1\ u \ v)A_i(0\ u'\ v')^{\top} = 0$ $(i=1,2,3)$
		implies $u' = v' = 0$. From this, we see that 
		there is no vanishing subspace $(X,Y)$ with 
		$(\dim X, \dim Y) = (2,1)$ or $(1,2)$. Therefore  
		mv-subspaces are 
		trivial ones $(K(\langle x \rangle)^3, 0)$ and $(0, K(\langle x \rangle)^3)$, 
		and $\ncrank A = 3$. 
		
		Next consider a weighted version
			\[
			A =  \left(  
			\begin{array}{ccc}
			0 &  t^{c_1}x_1 & t^{c_2} x_2 \\
			- t^{c_1} x_1& 0 & t^{c_3} x_3 \\
			- t^{c_2} x_2 & - t^{c_3}x_3 & 0
			\end{array} \right).
			\]
			for weights $c_1,c_2, c_3 \in \ZZ$.
	    Then, for $\alpha := (c_1+ c_2 -c_3,  c_1- c_2 + c_3, - c_1+ c_2 + c_3)$, it holds
	    \[
	    ( t^{- \alpha/2}) A   ( t^{- \alpha/2})   
		= \left(  
		\begin{array}{ccc}
			0 & x_1 & x_2 \\
			-x_1 & 0 & x_3 \\
			-x_2 & -x_3 & 0
		\end{array} \right).
		\]
		By Proposition~\ref{prop:opt_MVMP}, 
		we have $\deg \Det A = c_1 + c_2 + c_3$. On the other hand, 
		it obviously holds $\deg \det A = - \infty$. 
\end{Ex}

\subsection{Some classical examples in combinatorial optimization}\label{subsec:classical} 
As mentioned in the introduction, 
some of classical combinatorial optimization problems
are formulated as the computation of the degree of the determinant 
of a linear matrix with the rank-1 property.
Here we consider representative three examples (bipartite matching, 
linear matroid greedy algorithm, linear matroid intersection), and
explain how the steepest descent algorithm works on these problems.
This gives some new insights on classical algorithms 
in combinatorial optimization. 

For a subset $J \subseteq \{1,2,\ldots,n\}$, 
let $\QQ^J \subseteq \QQ^n$ denote the coordinate subspace spanned 
by unit vectors $e_i$ for $i \in J$, and let ${\bf 1}_{J} := \sum_{i \in J} e_i$.

\subsubsection{Bipartite matching}\label{subsub:bipartite}
Let $G = (U,V; E)$ be a bipartite graph with color classes $U,V$.
Vertices of $U$ (resp. $V$) are numbered as $1,2,\ldots,n$ (resp. $1,2,\ldots,m$).
As mentioned in the introduction, 
the maximum size $\nu(G)$ of a matching of $G$ is written as 
the rank of an $n \times m$ 
linear matrix $A = \sum_{e = ij \in E} x_e E_{ij}$, 
where $E_{ij}$ is the matrix having $1$ at $(i,j)$-entry and zero at others, 
and $x_e$ $(e \in E)$ are variables. 
Each $E_{ij}$ of $A$ has rank $1$. It holds that $\rank A =$ nc-$\rank A$. 
By Theorem~\ref{thm:FortinReutenauer}, 
$\nu(G)$ is equal to $n+m$ minus the dimension of an mv-subspace $(X,Y)$.
Observe that any feasible subspace $(X,Y)$ of MVSP
is of the form of $(\QQ^J, \QQ^K)$ for 
$J \subseteq \{1,2,\ldots,n\}, K \subseteq \{1,2,\ldots,m\}$ 
such that there is no edge between $J$ and $K$, i.e., 
$J \cup K$ is a stable set of $G$, and is the complement of a vertex cover.
Thus Theorem~\ref{thm:FortinReutenauer} is nothing but 
K\"onig's formula for the maximum matching.

Next we consider the weighted situation.
Suppose $|U| = |V| = n$ for simplicity, and that
each edge $e \in E$ has weight $c_e \in \ZZ$.
Consider a linear matrix $A := \sum_{e = ij \in E} t^{c_{e}} x_e E_{ij}$
over $\QQ (t)$.
Then the maximum weight of a perfect matching of $G$ 
is equal to $\deg \det A$, and is equal to $\deg \Det A$ 
(by Theorem~\ref{thm:rank-1_degdet}).
We explain how the steepest descent algorithm works in this case.
We use the modified step 3$'$.
Suppose for explanation that $c_e \leq 0$ for each $e \in E$. 
Linear matrix $A^0$ corresponds to the subgraph $G^0$ consisting of edges with $c_e = 0$. 
A steepest direction is given by $(\QQ^J, \QQ^K)$ 
for a maximum stable set $J \cup K$ of $G^0$. 
In step 3$'$, $\kappa$ is chosen as
the maximum of $- c_e (> 0)$ for edges $e \in E$ belonging to $J \cup K$. 
Then $A$ is updated to $(t^{\kappa {\bf 1}_J}) A (t^{- \kappa ({\bf 1}- {\bf 1}_K)})$.
SDA repeats this process,  
which is viewed as a cut-canceling algorithm.
The resulting optimal solution is a form of 
$(\langle(t^{p}) \rangle_{\rm R}, \langle(t^q)\rangle_{\rm L})$ for $p,q \in \ZZ^n$. 
Here vectors $p,q$ are dual optimal solutions of 
the LP-formulation of the weighted matching problem.
If we always choose a maximum stable set $J \cup K$ 
with minimal $J$ (and maximal $K$) in each iteration, then SDA coincides with the Hungarian method.
Indeed, $J$ is the subset reachable from 
vertices in $U$ not covered  by a maximum matching 
in the corresponding residual graph of $G^0$.
\subsubsection{Maximum weight base in linear matroid}
Let $a_1,a_2,\ldots,a_m$ be $n$-dimensional vectors of $\QQ^n$.
Consider $m$ variables $x_1,x_2,\ldots,x_m$, 
and linear matrix $A = \sum_{i=1}^m x_i a_i a_i^{\top}$.
Then $\rank A = \ncrank A = \rank (a_1\ a_2\ \cdots a_m)$.  
Let $W \subseteq \QQ^n$ be the vector space 
spanned by $a_1,a_2,\ldots,a_m$.
Then $(W^{\bot},\RR^n)$ is an mv-subspace, 
where $W^{\bot}$ denotes the orthogonal subspace of $W$.

As in Section~\ref{subsub:bipartite}, consider the weighted situation.
Let $c_i \in \ZZ$ be the weight on $a_i$ for each $i$.
Consider linear matrix $A = \sum_{i=1}^m t^{c_i} x_i a_i a_i^{\top}$.
Then $\deg \det A = \deg \Det A$ is equal to 
the maximum of $\sum_{i \in B} c_i$ over all $B \subseteq \{1,2,\ldots,m\}$
such that $\{ a_i \mid i \in B\}$ forms a basis of $\QQ^n$.
Namely $\deg \det A$ is equal to the maximum weight of a base of the matroid 
represented by vectors $a_1,a_2,\ldots,a_m$.

In this case, 
the steepest descent algorithm
is viewed as the greedy algorithm.
Suppose that each $c_i$ is nonpositive. 
Then $A^0$ is the linear matrix 
$\sum_{i \in I_0}  x_i a_i a_i^{\top}$, where 
$I_0$ is the set of indices $i$ with $c_i = 0$.
Consider the subspace $W_1$ 
spanned by $a_i$ $(i \in I_0)$.
Then $(W_1^{\bot},\QQ^n)$ is an mv-subspace.
Consider a nonsingular matrix $Q \in \QQ^{n \times n}$
such that the first $k_1$ rows form a basis of $W_1^{\bot}$.
In step 3, $A$ is updated as $(t^{{\bf 1}_{\leq k_1}}) Q A$, 
or feasible module 
$(L,M)$ moves from $(\langle I \rangle_{\rm L}, \langle I \rangle_{\rm R})$ 
to $(\langle (t^{{\bf 1}_{\leq k_1}})Q \rangle_{\rm L}, \langle I \rangle_{\rm R})$.
The exponent of term $t^{c_i} x_i Q a_ia_i^{\top}$ increases 
if and only if $a_i$ does not belong to $W_1$. 
Thus, in step 3$'$, 
SDA can augment $A$ as $(t^{\alpha_1 {\bf 1}_{\leq k_1}}) Q A$ 
until $c_i + \alpha_1$  becomes zero for some $a_i \not \in W_1$.
Then $I_0$ increases, 
and the next subspace $W_2$ 
spanned by $a_i$ $(i \in I_0)$ increases.
Consequently $W_2^{\bot} \subset W_1^{\bot}$. 
We can modify $Q$ so that it also includes a basis of $W_2^{\bot}$.
SDA moves $(L,M)$ to $(\langle (t^{\alpha_1 {\bf 1}_{\leq k_1} + \alpha_2 {\bf 1}_{\leq k_2}}Q) \rangle_{\rm L}, \langle I \rangle_{\rm R})$, and obtain $W_3^{\bot} \subset W_2^{\bot}$ as above.
Repeat the same process. Eventually SDA reaches 
an optimal module 
$(\langle (t^{\sum_{j=1}^h \alpha_j {\bf 1}_{\leq k_j}})Q \rangle_{\rm L}, \langle I \rangle_{\rm R})$, 
where $Q$ consists of bases of vector spaces
$W_1^{\bot} \supset W_2^{\bot} \supset \cdots \supset W_h^{\bot}$.
This process simply chooses vectors $a_i$ from largest weights. 
It is nothing but the matroid greedy algorithm, 
where we need no explicit computation of $Q$. 
The obtained $\alpha_k$ can be interpreted as 
an optimal dual solution of 
the LP-formulation of the maximum weight base problem, 
where $\alpha_k$ is the dual variable corresponding to the flat $\{i \mid a_i \in W_k\}$.
\subsubsection{Linear matroid intersection}
In addition to $a_1,a_2,\ldots,a_m$ above, 
we are given vectors $b_1,b_2,\ldots,b_m \in \QQ^n$.
Consider a linear matrix $A = \sum_{i=1}^m x_i a_i b_i^{\top}$ 
with variables $x_1,x_2,\ldots,x_m$.
We have $\rank A = \ncrank A$, they are equal to  
the maximum cardinality of a subset $I \subseteq \{1,2,\ldots,m\}$
such that both $\{ a_i \mid i \in I\}$ and $\{b_i \mid i \in I\}$ are
independent. Namely, $\rank A$ is the maximum cardinality of 
a common independent set
of two matroids ${\bf M}_1$ and ${\bf M}_2$ represented by $a_1,a_2,\ldots,a_m$ and $b_1,b_2,\ldots,b_m$, respectively. 
For $I \subseteq \{1,2,\ldots,m\}$,
let $\rho(I)$ and $\rho'(I)$ denote 
the dimension of vector spaces spanned by $a_i$ $(i \in I)$ and by $b_i$ $(i \in I)$, respectively.
By the matroid intersection theorem, 
$\rank A$ is the minimum of $\rho(I) + \rho'(J)$ over 
all bi-partitions $I,J$ of $\{1,2,\ldots,m\}$.
Then an mv-subspace $(X,Y)$ is given by 
$X = \{a_i \mid i \in I\}^{\bot}$ and $Y = \{b_j \mid j \in J\}^{\bot}$
for bi-partition $I,J$ attaining the minimum.
This fact is noted in~\cite{Lovasz89}.

Suppose that we are further given 
weights $c_i \in \ZZ$ for each $i=1,2,\ldots,m$.
Consider a linear matrix $A = \sum_{i=1}^m t^{c_i} x_i a_i b_i^{\top}$ over $\QQ(t)$.
Then $\deg \det A = \deg \Det A$ is equal to 
the maximum weight $\sum_{i \in B} c_i$ 
of a common independent set $B \subseteq \{1,2,\ldots,m\}$ with $|B| = n$ 
of matroids ${\bf M}_1$ and ${\bf M}_2$. 
Namely, the problem of finding $\deg \det A$ 
is the weighted linear matroid intersection problem.
Let us explain the behavior of 
the steepest descent algorithm applied to this case.
Suppose that 
we are given a feasible module $(L,M)$ of 
form $L = \langle (t^{\alpha}) S \rangle_{\rm L}$ and 
$M = \langle T (t^\beta) \rangle_{\rm R}$
for nonsingular matrices $S,T \in \QQ^{n \times n}$ 
and integer vectors $\alpha, \beta \in \ZZ^n$.
It may appear that
a naive choice of a steepest direction $(X,Y)$ at $(L,M)$
would violate this form in the next step, 
but, in fact, such a situation can naturally be avoided.

Let $R_1,R_2,\ldots,R_\mu$ be the partition of $\{1,2,\ldots,n\}$
such that $i,j$ belong to the same part if and only if $\alpha_i = \alpha_j$.
Similarly, let $C_1,C_2,\ldots,C_\nu$ be the partition 
such that $i,j$ belong to the same part if and only if $\beta_i = \beta_j$.
Regard matrix $SAT$ as a block matrix, where  
columns and rows are partitioned by $R_1,R_2,\ldots,R_\mu$ and $C_1,C_2,\ldots,C_\nu$.
In $(t^\alpha) SAT (t^\beta)$, 
the $(k,\ell)$-th block is uniformly 
multiplied by $t^{\alpha_i + \beta_j}$ for $i \in R_k$, $j \in C_\ell$.
Consider the linear matrix $((t^\alpha) SAT (t^\beta))^0$ 
(to obtain a steepest direction).
Then each summand $((t^\alpha) S a_ib_i^{\top} T (t^\beta))^{0}x_i$
has at most one nonzero block, where the nonzero block (if it exists) has rank $1$. 
Now $((t^\alpha) SAT (t^\beta))^0$ is essentially 
in the situation of a {\em partitioned matrix with rank-1 blocks}~\cite{HH16DM}.
See also Section~\ref{subsec:HH} in Appendix.
By the partition structure, 
any mv-subspace $(X,Y)$ is of the form of 
$(X_1 \oplus X_2 \oplus \cdots \oplus X_{\mu}, Y_1 \oplus Y_2 \oplus \cdots \oplus Y_{\nu})$ 
for $X_k \subseteq \QQ^{R_k}$ and $Y_\ell \subseteq \QQ^{C_\ell}$ 
$(k=1,2,\ldots,\mu,\ell = 1,2,\ldots,\nu)$.
Then basis matrices $S',T'$ for $X,Y$ 
are taken as block diagonal form so that
$S'(t^{\alpha}) = (t^{\alpha})S'$ and $T'(t^{\beta}) = (t^{\beta})T'$.
In the next iteration, 
$(L,M)$ is $(\langle (t^{\alpha'}) S'S \rangle_{\rm L}, \langle TT'(t^{\beta'}) \rangle_{\rm R})$.
Consequently the obtained optimal solution 
is of the form of $(\langle (t^{\alpha}) S \rangle_{\rm L}, \langle T(t^{\beta}) \rangle_{\rm R})$.
In particular, exponent vectors $\alpha$ and $\beta$ can  
be dealt with as numerical vectors. 
%

We here note that this algorithm is viewed as 
a variant of the primal dual algorithm
for weighted matroid intersection problem by Lawler~\cite{Lawler75}. 
His algorithm keeps and updates chains of flats 
in ${\bf M}_1$ and ${\bf M}_2$ and their weights.
Observe that module $\langle T (t^{\beta})_{\rm L} \rangle$
can be identified with a chain $\emptyset \neq X_1 \subset X_2 \subset \cdots \subset X_n = \QQ^n$ of subspaces and coefficients 
$\lambda_i$ $(i=1,2,\ldots,n)$ such that $\lambda_i \geq 0$ for $i < n$.
Indeed, arrange $\beta$ as $\beta_1 \geq \beta_2 \geq \cdots \geq \beta_n$, 
and define $X_i$ as the subspace spanned the first $i$ rows
and $\lambda_i$ as $\beta_i - \beta_{i+1}$ (with $\beta_{n+1} = 0$).
This correspondence is unique if subspaces $X_i$ with $\lambda_i = 0$ are omitted.
In this way, module $(L,M)$ can be kept as a pair of weighted chains of subspaces.
If these subspaces are orthogonal complements of 
the subspaces spanned by some subsets of $a_1,a_2,\ldots,a_m$ and $b_1,b_2,\ldots,b_m$, 
then $(L,M)$ can further be kept by a pair
of weighted chains of flats of matroids ${\bf M}_1$ and ${\bf M}_2$, 
as in Lawler's algorithm.

Moreover, if we choose an mv-subspace $(X,Y)$ 
with minimal $X$ (and maximal $Y$) in each iteration 
and use the modified step 3$'$, 
then SDA coincides with 
the {\em weight splitting algorithm} by Frank~\cite{Frank81} 
(applied to linear matroids), 
where $(\alpha, \beta)$ corresponds to a weight splitting; 
see also \cite[Section 13.7]{KorteVygen} for the weight splitting algorithm.
The detail of this correspondence is given in \cite{FurueHirai19}.

\subsection{Mixed polynomial matrix}
A {\em mixed polynomial matrix}
is a polynomial matrix $A =　\sum_{k=0}^\ell (Q_k + T_k) t^k$ with indeterminate $t$
such that $Q_k$ is a matrix over $\QQ$,
each entry of $T_k$ is zero or one of variables $x_1,x_2,\ldots,x_m$, and
each variable $x_i$ appears as one entry of one of $T_1,T_2,\ldots,T_k$.
In the case of $\ell =0$, $A$ is called a {\em mixed matrix}. 
See~\cite{MurotaMatrix} for detail of mixed (polynomial) matrices.
A mixed polynomial matrix
is viewed as a linear matrix over $\QQ(t)$ with rank-$1$ summands, 
since the coefficient matrix of $x_k$
is written as $E_{ij}$.
Therefore it holds that $\deg \det A = \deg \Det A$.

It is shown in~\cite{IwataOkiTakamatsu17,IwataTakamatsu13} that
the combinatorial relaxation algorithm computes $\deg \det A$ 
in $O(\ell^2 n^{\omega+2})$ time. 
This estimate seems very rough, since 
it is based on a trivial bound $\ell n$ 
of the number of iterations.
In the case of the steepest descent algorithm, 
we obtain a sharper estimate. 
\begin{Thm}\label{thm:mixedpoly}
	Let $A$ be an $n \times n$ mixed polynomial matrix with maximum degree $\ell$.
	By the steepest descent algorithm, $\deg \det A$ can be computed in 
	$O(\ell^2 n^{\omega+2})$ time,  
	and in $O((\ell - \alpha_n) n^3 \log n+  (\ell - \alpha_n)^2 n^{\omega} )$ time if $A$ is nonsingular, 
	where $\alpha_n$ is the minimum degree of diagonals of the Smith-McMillan form of $A$.
\end{Thm}
To prove Theorem~\ref{thm:mixedpoly}, 
we will work on a mixed matrix of a special form, as in \cite{IwataOkiTakamatsu17}.
A {\em layered mixed (polynomial) matrix}~\cite{MurotaBook}  
is a mixed (polynomial) matrix $A$ of form
\begin{equation}\label{eqn:LM0}
A = \left( \begin{array}{c}
Q \\
T
\end{array}\right),
\end{equation}
where $Q$ is a matrix over $\QQ[t]$ 
and $T$ is a variable matrix as above.
It is well-known in the mixed-matrix literature
that the rank and deg-det computation of a mixed matrix $Q+T$
reduce to those of a layered one
\begin{equation}\label{eqn:LM}
\left(\begin{array}{cc}
Q & I \\
T & D
\end{array}\right),
\end{equation}
where $D$ is a diagonal matrix of new variables.  

To compute a steepest direction, we need an mv-subspace of 
a layered  mixed (nonpolynomial) matrix, which is naturally obtained 
from a min-max formula of the rank.
Let $A$ in (\ref{eqn:LM0}) 
be an $n \times n'$ layered mixed matrix.
For $J \subseteq \{1,2,\ldots,n'\}$, 
let $Q[J]$ denote the submatrix of $Q$ consisting of $j$-th columns for $j \in J$, and
let $\varGamma(J)$ 
denote the set of indices $i$ with $T_{ij} \neq 0$ for some $j$.
\begin{Thm}[{\cite[Theorem 4.2.5]{MurotaMatrix}}]
	For an $n \times n'$ layered mixed matrix $A$ in $(\ref{eqn:LM0})$.
	\begin{equation}\label{eqn:LM-rk-formula}
	\rank A = \min_{J \subseteq \{1,2,\ldots,n'\}} 
	\{ \rank Q[J] + |\varGamma(J)| - |J| \} + n'.
	\end{equation}
\end{Thm}
Let $R_Q$ and $R_T$ 
denote the sets of row indices of matrices $Q$ and $T$, respectively.
\begin{Lem}\label{lem:J}
	Let $J$ be a minimizer of $(\ref{eqn:LM-rk-formula})$.
	Let $X := \ker_{\rm L} Q[J] \oplus \QQ^{R_T \setminus \varGamma(J)}$ 
	and $Y := \QQ^{J}$.
	Then $(X,Y)$ is an mv-subspace.
\end{Lem}
\begin{proof} This follows from 
$n + n' - \dim X - \dim Y = n + n' - |R_Q| + \rank Q[J] - |R_T| + |\varGamma(J)| - |J| =   \rank Q[J] + |\varGamma(J)| - |J| + n'$.	
\end{proof}

\begin{proof}[Proof of Theorem~\ref{thm:mixedpoly}]
	For a mixed matrix $Q+T$, we compute the degree of the determinant of
	the corresponding layered mixed matrix (\ref{eqn:LM}).
    As an initialization (step 0), 
    $(P, Q)$ is defined as $P = I$ and $Q = (t^{- \ell {\bf 1}_{\leq n}})$, 
    and let $A \leftarrow PAQ$. 
	In step 1, 
	SDA computes an mv-subspace of layered mixed matrix $A^0$.	
	A minimizer $J$ of (\ref{eqn:LM-rk-formula}) 
    is obtained by Cunningham's matroid intersection algorithm~\cite{Cunningham86} 
    in $O(n^3 \log n)$ time. Namely $\gamma = O(n^3 \log n)$.
    In step 3, the matrix multiplication is needed only for the $Q$-part of 
    $A$, which eliminates $m$ in the time complexity of Theorem~\ref{thm:degDet_algo}. 
\end{proof}

\paragraph{Application to DAE.}
A motivating application of mixed polynomial matrices 
is analysis of linear {\em differential algebraic equations (DAE)} 
with constant coefficients, where
each coefficient is an accurate number 
or one of (inaccurate) parameters $x_1,x_2,\ldots, x_m$, 
and no parameter appears as distinct coefficients; see \cite[Chapter 6]{MurotaBook}.
By the Laplace transformation,  
the analysis of such a DAE reduces to
linear equation $A x = b$, where 
$A$ is a mixed polynomial matrix over $\RR[s]$
with variables $x_1,x_2,\ldots,x_m$. 
Suppose the case where matrix $A$ is a square matrix.
The {\em index} is a barometer of ``difficulty" of DAE $Ax = b$, and is defined as $- \alpha_n + 1$, 
where $\alpha_n$ is the minimum degree of the Smith-McMillan form of $A$.
A DAE with high index ($\geq 2$) is difficult to solve numerically, 
and suggests an inconsistency of the mathematical modeling in deriving this DAE.
Therefore it is meaningful to decide whether the index of given a DAE is at most 
the limit $\varDelta$. 
Here $\varDelta (\simeq 2)$
is the allowable upper bound for 
the index of DAE-models of the system we want to analyze.
The steepest descent algorithm can decide 
in $O((\ell + \varDelta) n^3 \log n + (\ell+ \varDelta)^2 n^{\omega})$ time
whether DAE $Ax = b$ has index at most $\varDelta$. 
Indeed, apply SDA to $A$.
Index $- \alpha_n + 1$ is obtained from 
the number $\ell - \alpha_n$ of required 
iterations (Lemma~\ref{lem:iterations}).
If SDA terminates before $\ell + \varDelta$ iterations,
then the DAE has index within $\varDelta$.
Otherwise the index is over the limit $\varDelta$.

\section*{Acknowledgments}
The author thanks Kazuo Murota, Satoru Iwata, Yuni Iwamasa,  Taihei Oki  and Koyo Hayashi 
for careful reading and helpful comments, and thanks Mizuyo Takamatsu for remarks.
Also the author thanks the referees for helpful comments.
This work was supported by JSPS KAKENHI Grant Numbers 
JP26280004,  JP17K00029.

\appendix
\section{Appendix}
\subsection{Relation to the formulation by Hamada and Hirai}\label{subsec:HH}
	Hamada and Hirai~\cite{HamadaHirai17} actually studied 
	the following variant of MVSP: We are given a matrix $A \in \KK^{m \times n}$ 
	partitioned into submatrices as
	\[
	A = \left(
	\begin{array}{ccccc}
	A_{11} & A_{12} &\cdots & A_{1\nu} \\
	A_{21} & A_{22} &\cdots & A_{2\nu} \\
	\vdots & \vdots & \ddots & \vdots \\
	A_{\mu1}&A_{\mu2} &\cdots & A_{\mu \nu}
	\end{array}\right),
	\]
	where $A_{\alpha \beta} \in \KK^{m_\alpha \times n_\beta}$ 
	for $\alpha=1,2,\ldots,\mu$ and
	$\beta=1,2,\ldots,\nu$. 
	The goal is to find a collection of vector subspaces 
	$X_\alpha \subseteq \KK^{m_{\alpha}}$, 
	$Y_{\beta}  \subseteq \KK^{n_{\alpha}}$ 
	$( \alpha=1,2,\ldots,\mu;  \beta=1,2,\ldots,\nu)$ such that
	\begin{equation}\label{eqn:vanishing}
	A_{\alpha \beta} (X_{\alpha}, Y_{\beta}) = \{0\} \quad 
	(\alpha=1,2,\ldots,\mu; \beta=1,2,\ldots,\nu), 
	\end{equation}
	and the sum of their dimensions
	\begin{equation}\label{eqn:dimension}
	\sum_{\alpha = 1}^{\mu} \dim X_{\alpha} + \sum_{\beta = 1}^{\nu} \dim Y_{\beta} 
	\end{equation}
	is maximum. 
	Hamada and Hirai refer to this problem as MVSP. 
	We here call it {\em block-MVSP}
	
	Block-MVSP  reduces to MVSP (in our sense).
	Indeed, 
	introduce a new variable $x_{\alpha \beta}$ for each $\alpha,\beta$, 
    and consider MVSP for the linear matrix
	\[
	\tilde A = \left(
	\begin{array}{ccccc}
	A_{11}x_{11} & A_{12}x_{12} &\cdots & A_{1\nu}x_{1\nu} \\
	A_{21}x_{21} & A_{22}x_{22} &\cdots & A_{2\nu}x_{2\nu} \\
	\vdots & \vdots & \ddots & \vdots \\
	A_{\mu1}x_{\mu1}&A_{\mu2}x_{\mu2} &\cdots & A_{\mu \nu}x_{\mu \nu}
	\end{array}\right),
	\]
	Observe that any mv-subspace $(X,Y)$ is necessarily
	a form of $(X_1 \oplus X_2 \oplus \cdots \oplus X_{\mu}, Y_1 \oplus Y_2 \oplus \cdots \oplus Y_{\nu})$, 
	where $X_{\alpha}, Y_{\beta}$ satisfy (\ref{eqn:vanishing}) and maximize (\ref{eqn:dimension}).
	(Consider the projections $X_{\alpha}$ of $X$ 
	and $Y_{\beta}$ of $Y$ to the coordinate subspaces corresponding to the partitions. Then 
	$(\oplus_{\alpha} X_\alpha, \oplus_{\beta} Y_\beta)$ is 
	also feasible to MVSP 
	with $X \subseteq \oplus_{\alpha} X_\alpha$ and $Y \subseteq \oplus_{\beta} Y_\beta$.)
	
	The converse reduction is also possible.
	For a linear $n \times n'$ matrix $A = A_0 + \sum_{i=1}^{m} A_i x_i$, consider the block matrix
	\begin{equation}\label{eqn:converse}
	\left(
	\begin{array}{ccccc}
	A_{0} & I &    &   \\
	A_{1} & I & \ddots &   \\
	\vdots &   & \ddots & I \\
	A_{m} &  &    & I
	\end{array}\right),
	\end{equation}
	where the unfilled blocks are 
	zero matrices and $I$ is the $n$ by $n$ unit matrix.
	For a solution $(X,Y)$ of MVSP for $A$, 
	the collection 
	of subspaces $X_{\alpha}$, $Y_{\beta}$ defined by
	\begin{eqnarray*}
		&& X_{\alpha} := X \quad (\alpha =1,\ldots, m+1), \\
		&& Y_1 := Y,\ Y_{\beta} := X^{\bot} \quad (\beta =2,\ldots, m+1) 
	\end{eqnarray*}
	is a solution of the block-MVSP for the matrix (\ref{eqn:converse}), and has the sum of dimensions
	\[
	(m+1) \dim X + \dim Y + m (n- \dim X) = \dim X + \dim Y + mn.
	\]
	
	In fact, we can always choose an optimal solution with such a form in this block-MVSP.
	Let $X_{\alpha}, Y_{\beta}$ ($\alpha,\beta =1,\ldots,m+1$) be an optimal solution of the block-MVSP.
	Necessarily $Y_{\alpha+1} = X_{\alpha}^{\bot} \cap X_{\alpha+1}^{\bot}$ 
	holds for $\alpha = 1,2,\ldots,m$. 
	Then it holds
	\begin{eqnarray*}
		&& \dim X_{\alpha} + \dim X_{\alpha+1} + \dim Y_{\alpha+1}  \\
		&& = \dim X_{\alpha} \cap X_{\alpha+1} + \dim (X_{\alpha} + X_{\alpha+1})  + \dim  X_{\alpha}^{\bot} \cap X_{\alpha+1}^{\bot} \\
		&& = \dim X_{\alpha} \cap X_{\alpha+1} + \dim (X_{\alpha} + X_{\alpha+1})  + \dim  (X_{\alpha} + X_{\alpha+1})^{\bot} \\
		&& = \dim X_{\alpha} \cap X_{\alpha+1} + n \\
		&& = 2 \dim X_{\alpha} \cap X_{\alpha+1} + \dim (X_{\alpha} \cap X_{\alpha+1})^{\bot}.
	\end{eqnarray*}
	Hence $(X_{\alpha}, X_{\alpha+1}, Y_{\alpha+1})$ 
	can be replaced by $( X_{\alpha} \cap X_{\alpha+1}, X_{\alpha} \cap X_{\alpha+1},  (X_{\alpha} \cap X_{\alpha+1})^{\bot})$.
	This implies the existence of an optimal solution in which all $X_{\alpha}$ 
	are equal.
	In particular, $(X,Y) := ( X_1 \cap \cdots \cap X_{m+1},Y_1)$ is an mv-subspace of MVSP for $A$.
	
	As was noted in \cite[Remark 3.13]{HamadaHirai17}, 
	without such a reduction, 
	their approach and algorithm
	are quickly adapted to MVSP (in our sense).

\subsection{Proof of Lemma~\ref{lem:L(F(t)^n)}}
    Let ${\cal L} := {\cal L}_{\rm R}(\FF(t)^n)$. 
    We omit ${\rm R}$ from $\langle \cdot \rangle_{\rm R}$.
    We first consider the join and meet of two $L, M \in {\cal L}$.
    Suppose that  $L = \langle P \rangle$ and $M = \langle Q \rangle$. 
    Then $L \subseteq M$ if and only if $Q^{-1}P$ is proper.
    From we see that $t^{- \ell} L \subseteq L \cap M \subseteq L + M \subseteq t^\ell L$ for a large $\ell> 0$.
    Both $L \cap M$ and $L + M$ are submodules of 
    the free module $t^{\ell}L$ over PID $\FF(t)^-$.
    Hence both $L \cap M$ and $L + M$ are free. 
    Since they contain the full-rank free module $t^{- \ell} L$, they are also full-rank, and hence  belong to ${\cal L}$.  
    Necessarily $L \wedge M = L \cap M$ and $L \vee M = L + M$.
      
     We show the modularity $L + (M \cap L') = (L + M) \cap L'$ 
     for $L' \in {\cal L}$ with $L \subseteq L'$.  
     It suffices to show $(\supseteq)$.
     Let $z = x + y \in (L+ M) \cap L'$ with $x \in L$ and $y \in M$.
    By $z,x \in L'$, it holds $y = z - x \in L'$, and $y \in M \cap L'$. 
    Thus  $z = x+y \in L + (M \cap L')$, as required.
    Note that this argument is standard for proving that normal subgroups of a group forms a modular lattice.
    
    We next show that $\deg$ is a unit valuation.
    Here we can assume that
    $M = \langle Q \rangle = \langle P (t^{\alpha}) \rangle$, 
    where $(t^{\alpha})$ is the Smith-McMillan form of $P^{-1} Q$. 
    Indeed, If $S^{-1}P^{-1}QT = (t^{\alpha})$ for biproper $S,T$, 
    then $P$ and $Q$ can be replaced by $PS$ and $QT$, respectively, since $\langle P \rangle =\langle PS \rangle$ and $\langle Q \rangle =\langle QT \rangle$. 
    Therefore, if 
    $L$ is covered by $M$, then $\alpha = e_1$, and 
    $\deg M - \deg L = \deg \Det P(t^{e_1})- \deg \Det P = 1$．
    
    Also $\langle t P \rangle$ is the join of 
    $\langle P (t^{e_i}) \rangle$ for $i=1,2,\ldots,n$, 
    and the ascending operator is given by $L \mapsto t L$, 
    which is clearly an automorphism on ${\cal L}$.
    This concludes that ${\cal L}$ is a uniform modular lattice.

\subsection{Proof of Lemma~\ref{lem:apartment}}
  We continue the above notation of omitting ${\rm R}$.
  Consider two short chains $C,D$.
  By (B1) in Lemma~\ref{lem:skeleton}, 
  there is a $\ZZ^n$-skeleton $\varSigma$ containing them.
  Identify $\varSigma$ with $\ZZ^n$. 
  If $L \in \varSigma$ corresponds to $x \in \ZZ^n$, then we write $L \equiv x$.
  We may assume that $C,D$ belong to interval $[0,k]^n \subseteq \ZZ^n$.
  It suffices to show that the interval $[0,k]^n$ belongs to $\varSigma(P)$ 
  for some nonsingular matrix $P$.
  Consider $L \in \varSigma$ with $L \equiv {\bf 0} \in \ZZ^n$.
  We show by an inductive argument 
  that there are $p_1,p_2,\ldots,p_n \in L$ 
  such that $L = \langle (p_1\ p_2 \cdots\ p_n) \rangle$ 
  and  $\langle(p_1\ \cdots\ t^{\ell} p_{i}\ \cdots\ p_n) \rangle \equiv \ell e_i$ for each $i$ and $\ell \leq k$.
  Then $P = (p_1\ p_2\ \cdots\ p_n)$ is a desired matrix.
  Indeed, for  $\alpha \in [0,k]^n$, 
  the join $M$ of $\langle(p_1\ \cdots\ t^{\alpha_i} p_{i}\ \cdots\ p_n)\rangle$ over $i=1,2,\ldots, n$ corresponds to $\alpha$.
  Obviously $\langle P (t^{\alpha}) \rangle \subseteq M$.
  Since $M$ and $\langle P (t^{\alpha}) \rangle$ have 
  the same height $\deg L + \alpha_1 + \alpha_2 + \cdots + \alpha_n = \deg \det P (t^{\alpha})$, it must hold 
  $\langle P (t^{\alpha})  \rangle = M \equiv \alpha$.
  Thus the interval $[0,k]^n$ belongs to $\varSigma(P)$.

  Suppose that $L = \langle (q_1\ q_2 \cdots\ q_n) \rangle$. 
  By Lemma~\ref{lem:[L,tL]}~(4), 
  for each $i$ there is $p_i \in L$ such that 
  $e_i \equiv p_i t \FF(t)^- + L$. 
  Then $p_1,p_2,\ldots,p_n$ form an $\FF(t)^-$-basis of $L$, 
  i.e., $L = \langle (p_1\ p_2 \cdots\ p_n) \rangle$. 
  Indeed, 
  by $\sum_{i} p_i t \FF(t)^- + L \equiv e_1 + e_2 + \cdots + e_n = {\bf 1} \equiv t L$,
  each $tq_j$ is written as an $\FF(t)^-$-linear combination 
  of $tp_1,tp_2,\ldots,tp_n$ and $q_1,q_2,\ldots,q_n$.
  Consequently, $(q_1\ q_2\ \cdots\ q_n )(I + t^{-1}A) = (p_1\ p_2\ \cdots \ p_n)B$ holds for some square matrices $A,B$ over $\FF(t)^-$.
  Here $I + t^{-1}A$ is biproper (Lemma~\ref{lem:key}).
  This means that each $q_j$ is written as an 
  $\FF(t)^-$-linear combination of $p_1,p_2,\ldots,p_n$.
  Then $L = \langle (p_1\ p_2 \cdots\ p_n) \rangle$ 
  and $e_i \equiv  \langle (p_1\ \cdots\ tp_i\ \cdots\ p_n) \rangle = p_i t \FF(t)^- + L$ for each $i$.

  Suppose (by induction) that $\langle (t^\ell p_1\ p_2\ \cdots\ p_n) \rangle \equiv \ell e_1$ holds for $\ell \leq k-1 (\geq 1)$.
  We show that $p_1$ can be replaced by some $p'_1$
  such that  $\langle (t^\ell p'_1\ p_2\ \cdots\ p_n) \rangle \equiv \ell e_1$ holds for $\ell \leq k$.
  Consider $L' \in \varSigma$ with $L' \equiv k e_1$.
  By Lemma~\ref{lem:[L,tL]}~(4), 
  $L'$ is generated by vectors 
  obtained by replacing one of $t^{k-1} p_1,p_2,\ldots, p_n$
  with $r = t (t^{k-1} p_1 \lambda_1 + \sum_{i=2}^{n} p_i \lambda_i)$
  for $\lambda_i \in \FF$.
  Now $\lambda_1 \neq 0$. Indeed, if $\lambda_1 =0$ and $\lambda_2 \neq 0$ (say), then
  $L' = \langle (t^{k-1} p_1\ r\ \cdots\ p_n) \rangle \subseteq 
  t \langle(t^{k-2} p_1\ p_2\ \cdots\ p_n)\rangle$, 
  implying a contradiction $k e_1 \leq (k-2) e_1 + {\bf 1}$.
  Let $p'_1 := t^{-k} r$.
  Thus $L' = \langle (t^{k} p'_1\ p_2\ \cdots\ p_n)\rangle$.
  Also observe $\langle (t^{\ell} p'_1\ p_2\ \cdots\ p_n) \rangle 
  = \langle (t^{\ell} p_1\ p_2\ \ldots\  p_n) \rangle$ for $\ell \leq k-1$.
  Replace $p_1$ by $p'_1$, and apply the same replacement 
  for $p_2,p_3,\ldots,p_n$. Then we obtain 
  desired $p'_1,p'_2,\ldots,p'_n$.

\subsection{Proof of Proposition~\ref{prop:SmithMcMillan}}\label{app:SM}

The degree of the determinant $\deg \Det$ is a {\em matrix valuation} 
in the sense of \cite[Section 9]{CohnSkewField} 
(with min and max interchanged); 
see Theorem 9.3.4 of the reference.
In particular, the $\deg \Det$ function satisfies the following property
((MV.4) in \cite[Section 9.3]{CohnSkewField}):
\begin{description}
	\item[{\rm (MV)}] For nonsingular $A \in \FF(t)^{n \times n}$ 
	and a vector $b \in \FF(t)^n$ regarded as a row (or column) vector, 
	let $B$ be the matrix obtained from $A$ by  replacing 
	the first row (or column) by $b$, and
	let $C$ be the matrix obtained from $A$ by adding $b$ to the first row (or column) vector.	
	Then it holds 
	\[
	\deg \Det C \leq \max \{ \deg \Det A, \deg \Det B \}.
	\] 
	The strict inequality holds only if $\deg \Det A = \deg \Det B$. 
\end{description}
In \cite{CohnSkewField}, only the column version  
is proved but the row version can be proved in the same way.
Indeed, 
	by column permutation, we can make $A$ (and $B,C$) so that 
	the cofactor $A'$ of $(1,1)$-entry is nonsingular.
	Then we have
	\begin{equation*}
	A = 
	\left(\begin{array}{cc}
	a_{11}  & a'  \\
	0 & A' 
	\end{array}\right)E, \ 
	B = 
	\left(\begin{array}{cc}
	b_{11}  & b'  \\
	0 & A' 
	\end{array}\right)E,\ 
	C = 
	\left(\begin{array}{cc}
	c_{11}  & c'  \\
	0 & A' 
	\end{array}\right)E,  
	\end{equation*}
	where $E$ is the product 
	of permutation matrices and upper unitriangular matrices, 
	and $(c_{11}\ c') = (a_{11}\ a') + (b_{11}\ b')$.
	Thus $\deg \Det A = \deg a_{11} + \deg \Det A'$,
	$\deg \Det B = \deg b_{11} + \deg \Det A'$, and
	$\deg \Det C = \deg c_{11} + \deg \Det A' = \deg (a_{11}+ b_{11}) + \deg \Det A'$.
	From $\deg (a_{11} + b_{11}) \leq \max \{ \deg a_{11}, \deg b_{11} \}$, we obtain (MV).

Now 
let us start to prove Proposition~\ref{prop:SmithMcMillan}.
For $u \in \FF(t)$ and $k,\ell \in \{1,2,\ldots,n\}$ with $k \neq \ell$, 
define $E(k,\ell;u) \in \FF(t)^{n \times n}$ by
\begin{equation*}
	E(k,\ell;u)_{ij} = \left\{
	\begin{array}{ll}
	1 & {\rm if}\ i=j, \\
	u & {\rm if}\ i=k, j = \ell, \\
	0 & {\rm otherwise}. 
	\end{array}\right.
	\end{equation*}
	$E(k,\ell;u)$ is called an {\em elementary matrix}.
	Observe that $E(k,\ell;u)$ is nonsingular 
	with $E(k,\ell;u)^{-1} = E(k,\ell;-u)$.
	In particular, $E(k,\ell;u)$ is biproper if and only
	if $u \in \FF(t)^-$.
	
	A required diagonalization is obtained as follows.
	First, by multiplying permutation matrices to the left and the right of $A$,
	modify $A$ so that $A_{{11}}$ has the maximum degree among all entries of $A$.
	By multiplying elementary matrices $E(1,\ell;u)$ from right and $E(\ell',1; u')$ from left, 
	modify $A$ so that all entries except $A_{11}$ in the first row and column 
	are zero. Here $u,u'$ can be taken from $\FF(t)^-$ by the maximality.
	Therefore $E(1,\ell;u)$ and $E(\ell,1';u')$ are biproper, 
	and the degree of entries of $A$ does not increase. 
	Now $A_{11}$ is written as $t^{\alpha_1} v$ for $\alpha_1 = \deg A_{11}$ and
	$v \in \FF(t)^-$ with $\deg v = 0$. 
	Multiply a biproper diagonal matrix 
	whose $(1,1)$-entry is $v^{-1}$ and other diagonals are $1$.
	Then $A_{11}$ is now $t^{\alpha_n}$.  
	Repeat the same process to the submatrix from the second row and column.
	Eventually $A$ is diagonalized to $(t^{\alpha})$ with $\alpha_1 \geq \alpha_2 \geq \cdots \geq \alpha_n$.
	By construction, $P,Q$ are the product of proper elementary matrices 
	and permutation matrices, and are biproper.

	Next we show that $\delta_k$ is invariant throughout the above procedure, which implies the latter part of the claim.
	It is obvious that $\delta_k$ is invariant under any row and column permutation.
	Consider the case of multiplying elementary matrix $E(i,j;u)$ from the right.
	This operation corresponds to 
	adding the $i$-th columns multiplied by $u$ to the $j$-th column.
	Consider a $k \times (k+1)$ submatrix
	having the $i$- and $j$-th columns, 
	and consider the change of its $k \times k$ minors by the multiplication of $E(i,j;u)$.
	Obviously any $k \times k$ minor containing the $i$-th column does not change.
	Consider the $k \times k$ minor not containing the $i$-th column. 
	From the property (MV), the degree of this minor 
	is at most the degree of the original or $\deg u (\leq 0)$ plus the degree of 
	the minor not containing $j$. From this, we see
	that the maximum degree of a $k \times k$ minor of this matrix does not change.
	Consequently $\delta_k$ does not change.
	The proof for the left multiplication is the same.
\end{document}